\documentclass[a4paper,11pt]{amsart}
\usepackage{amssymb}
\usepackage{mathrsfs}
\usepackage{amsmath,amssymb,amsthm,latexsym,amscd,mathrsfs}
\usepackage{indentfirst}
 \newcommand{\ROM}[1]{\mathrm{\uppercase\expandafter{\romannumeral#1}}}
  \theoremstyle{definition}

 \newtheorem{thm}{Theorem}[section]
 \newtheorem{lem}{Lemma}[section]
 \newtheorem{defn}{Definition}[section]
 \newtheorem{cor}{Corollary}[section]
  
 \newtheorem{rem}{Remark}[section]
 \newtheorem{prop}{Proposition}[section]

\newtheorem{ack}{Acknowledgements}   

\title[A Filtration for Isoparametric Hypersurfaces]{\textbf{A Filtration for Isoparametric Hypersurfaces in Riemannian Manifolds}}
\author[J. Q. Ge]{Jianquan Ge}\address{School of Mathematical Sciences, Laboratory of Mathematics and Complex Systems, Beijing Normal
University, Beijing 100875, China}\email{jqge@bnu.edu.cn}
\thanks {The project is partially supported by the NSFC (No.11071018, No.11001016, No.11331002 and No.11301027), the SRFDP (No.20100003120003 and No.20130003120008), and the Program for Changjiang Scholars and Innovative
Research Team in University.}
\thanks { The third author is the corresponding author. }
\author[Z. Z. Tang]{Zizhou Tang}\address{School of Mathematical Sciences, Laboratory of Mathematics and Complex Systems, Beijing Normal
University, Beijing 100875, China}\email{zztang@bnu.edu.cn}
\author[W. J. Yan]{Wenjiao Yan}
\address{School of Mathematical Sciences, Laboratory of Mathematics and Complex Systems, Beijing Normal
University, Beijing 100875, China} \email{wjyan@bnu.edu.cn}

\date{}
\begin{document}

\maketitle
\begin{abstract}
This paper introduces the notion of $k$-isoparametric hypersurface
in an $(n+1)$-dimensional Riemannian manifold for $k=0,1,...,n$.
Many fundamental and interesting results ( towards the
classification of homogeneous hypersurfaces among other things ) are
given in complex projective spaces, complex hyperbolic spaces, and
even in locally rank one symmetric spaces.
\end{abstract}

\section{Introduction}
A smooth non-constant function $f: N\rightarrow \mathbb{R}$ defined on a
Riemannian manifold $N$ is called \emph{transnormal} if there is a
smooth function $b:\mathbb{R}\rightarrow\mathbb{R}$ such that
\begin{equation}\label{iso1}
|\nabla f|^2=b(f),
\end{equation}
where $\nabla f$ is the gradient of $f$. If in addition there is a
continuous function $a:\mathbb{R}\rightarrow\mathbb{R}$ such that
\begin{equation}\label{iso2}
\triangle f=a(f),
\end{equation}
where $\triangle f$ is the Laplacian of $f$, then the function $f$
is called \emph{isoparametric} (cf. \cite{Wa87},
\cite{GT09},\cite{GT10}). \'{E}lie Cartan (cf.
\cite{Ca38,Ca39,Ca3,Ca4}) pointed out: equation (\ref{iso1}) means
that the level hypersurfaces $M_t:=f^{-1}(t)$ (where $t$ are regular
values of $f$) are parallel and equation (\ref{iso2}) further
implies that these parallel hypersurfaces have constant mean
curvatures.

The preimage of the global maximum (resp. minimum) of an isoparametric (or
transnormal) function $f$ is called the \emph{focal variety} of $f$,
denoted by $M_{+}$ (resp. $M_{-}$), if nonempty. A fundamental
structural result established by \cite{Wa87} asserts that each focal variety
of a transnormal function is a smooth submanifold (may be
disconnected and have different dimensions), and each connected
component $P_t$ of $M_t$ is a tube (tubular
hypersurface) or a ``half-tube" (when $codim(P)=1$) of the same
radius around a connected component $P$ of $M_{\pm}$. A hypersurface
$M$ in $N$ is called \emph{isoparametric} if it is a level
hypersurface of some locally defined isoparametric function $f$ on
$N$. Therefore, a hypersurface $M$ in $N$ is an isoparametric
hypersurface if and only if all its nearby parallel hypersurfaces
have constant mean curvatures, which is a local property with respect
to the ambient space $N$ in general, while isoparametric functions
are global objects of $N$ that may restrict strongly the
geometry and topology of the ambient space $N$, as stated in \cite{GT09}.


The theory of isoparametric functions (hypersurfaces) originated
from studies of hypersurfaces in real space forms with constant
principal curvatures (see \cite{Th00}, \cite{Ce} for excellent
surveys). As \'{E}lie Cartan (cf. \cite{Ca38,Ca39,Ca3,Ca4})
asserted, a hypersurface in a real space form has constant principal
curvatures if and only if all its nearby parallel hypersurfaces have
constant mean curvatures, thus it is an isoparametric hypersurface
defined as before. Isoparametric hypersurfaces in Euclidean or
hyperbolic space were easily classified due to the celebrated Cartan
identity. However, it turns out that isoparametric hypersurfaces in
the unit spheres are more complicated and plentiful, and thus have
not been completely classified up to now (for the newest progress,
please see \cite{GH}, \cite{CCJ07}, \cite{Imm08}, \cite{Chi2},
\cite{DN1}, \cite{Miy12}, \cite{TY13}, etc.). In the early 1970's,
M\"{u}nzner \cite{Mu80} produced a far-reaching generalization of
Cartan's work. He showed that an isoparametric hypersurface in a
sphere $S^{n+1}$ is an open part of a level hypersurface, say $M$,
of an isoparametric function $f$ which is the restriction to
$S^{n+1}$ of a Cartan polynomial $F$. By a \emph{Cartan polynomial}
( or \emph{isoparametric polynomial} ), we mean a homogeneous
polynomial $F$ on $\mathbb{R}^{n+2}$ satisfying the
Cartan-M\"{u}nzner equations
\begin{eqnarray}
  &&|\nabla F|^2 =g^2|x|^{2g-2}, \quad
  x\in\mathbb{R}^{n+2}, \label{eq1.3}\\
  &&\Delta F
  =\frac{g^2}{2}(m_2-m_1)|x|^{g-2}\label{eq1.4},
\end{eqnarray}
where $\nabla F$, $\Delta F$ denote the gradient and Laplacian of
$F$ on $\mathbb{R}^{n+2}$ respectively, and $m_1$, $m_2$ the
multiplicities of the maximal and minimal principal curvatures of
$M$, $g=deg(F)$ the number of distinct principal curvatures of
$M$. Further, using an elegant topological method M\"{u}nzner proved the
remarkable result that the number $g$ must be $1$, $2$, $3$, $4$,
or $6$.

Note that the Cartan-M\"{u}nzner equations
(\ref{eq1.3})-(\ref{eq1.4}) of the isoparametric polynomial $F$ on
$\mathbb{R}^{n+2}$ correspond to the equations
(\ref{iso1})-(\ref{iso2}) of the isoparametric function $f$ on
$S^{n+1}$ with the following equalities:
\begin{equation}\label{CM-isop-f}
b(f)=g^2(1-f^2),\quad\quad a(f)=\frac{g^2}{2}(m_2-m_1)-g(n+g)f,
\end{equation}
which only mean that the level hypersurfaces $M_t:=f^{-1}(t)$ have
constant mean curvatures. On the other hand, due to Cartan's result,
the level hypersurfaces $M_t$ essentially have constant principal
curvatures and hence constant mean curvatures of each order  (which
are elementary symmetric polynomials of principal curvatures). This
fantastic phenomenon suggests that there are hidden $n-1$ more
equations describing the constancy of higher order mean curvatures
of an isoparametric hypersurface in a sphere for the isoparametric
function $f$ (resp. isoparametric polynomial $F$) besides the
equations (\ref{iso1})-(\ref{iso2}) (resp. Cartan-M\"{u}nzner
equations (\ref{eq1.3})-(\ref{eq1.4})). It is this observation that
stimulates us to exhibit these hidden equations (see Theorem
\ref{implicit CM eqs}) which should possibly be helpful to provide a
geometric or an algebraic proof of M\"{u}nzner's remarkable result
on $g$ mentioned above.

Observing that there do exist isoparametric hypersurfaces in complex projective
spaces with non-constant principal curvatures (cf. \cite{Wa82}), we
will be concerned with the isoparametric functions (resp.
hypersurfaces) on Riemannian manifolds satisfying these hidden
equations (resp. more constant higher order mean curvatures ). This treatment will
filter isoparametric functions (resp. isoparametric
hypersurface) by \emph{$k$-isoparametric functions} (resp.
\emph{$k$-isoparametric hypersurfaces}) on a Riemannian manifold
$N^{n+1}$ for $k=1,\cdots,n$.


We now set up some notations. First of all, for an $n$ by $n$ real symmetric
matrix (or self-dual operator) $A$ with $n$ real eigenvalues
$(\mu_1,\cdots,\mu_n)=:\mu$ and $k=1,\cdots,n,$ we denote by
$\sigma_k(A)=\sigma_k(\mu)$ the $k$-th elementary symmetric
polynomial of $\mu$, \emph{i.e.,}
\begin{equation}\label{sigmak}
\sigma_k(A)=\sigma_k(\mu)=\sum_{i_1<\cdots<i_k}\mu_{i_1}\cdots\mu_{i_k}=\sum_{i_1<\cdots<i_k}A\left(^{i_1\cdots
i_k}_{i_1\cdots i_k}\right),\quad \sigma_0(A)=\sigma_0(\mu)\equiv1,
\end{equation}
where $A\left(^{i_1\cdots i_k}_{i_1\cdots i_k}\right)$'s are the
principal $k$-minors of $A$; and denote by $\rho_k(A)=\rho_k(\mu)$ the
$k$-th power sum, \emph{i.e.,}
\begin{equation}\label{rhok}
\rho_k(A)=\rho_k(\mu)=\sum_{i=1}^n\mu_i^k=tr(A^k),\quad\rho_0(A)=\rho_0(\mu)\equiv
n.
\end{equation}
In these notations, the Newton's identities can be stated as
\begin{equation}\label{Newton identity}
k\sigma_k=\sum_{i=1}^k(-1)^{i-1}\sigma_{k-i}\rho_i,\quad\quad for~~
k=1,\cdots,n,
\end{equation}
which show in particular that for $k=1,\cdots,n$,
\begin{equation}\label{sigkrhok}
 \{\sigma_1,\cdots,\sigma_k\}~ \emph{are
constant if and only if}~ \{\rho_1,\cdots,\rho_k\}~ \emph{are
constant.}
\end{equation}
Next, on a Riemannian manifold $N^{n+1}$, we define a sequence of partial
differential operators $\{\triangle_1,\cdots,\triangle_{n+1}\}$ over
$C^{\infty}(N^{n+1})$ by
\begin{equation}\label{laplacek}
\triangle_kf:=\sigma_k(H_f), \quad for~~k=1,\cdots,n+1,
\end{equation}
where $H_f$ is the Hessian of $f$ on $N^{n+1}$. It is respectively
the Laplacian and the Monge-Amp\`{e}re operator when $k=1$ and
$k=n+1$. Note that $\triangle_k$ is nonlinear when $k\geq2$.


\begin{defn}\label{defn-kisop}
For $1\leq k\leq n$, a non-constant smooth function $f$ on a
Riemannian manifold $N^{n+1}$ is called \emph{k-isoparametric},
 if $f$ is a transnormal
function satisfying equation (\ref{iso1}), and in addition there
exist continuous functions $a_1,\cdots,a_k\in C(\mathbb{R})$, such
that
\begin{equation}\label{k-isop-a}
\triangle_if=a_i(f), \quad for~~i=1,\cdots,k.
\end{equation}
We denote by $\mathscr{I}_k(N^{n+1})$ the set consisting of $k$-isoparametric functions on $N^{n+1}$.
A hypersurface $M^n$ in $N^{n+1}$ is called \emph{k-isoparametric}
if it is a level hypersurface of some locally defined
$k$-isoparametric function $f$ on $N^{n+1}$.
\end{defn}


\begin{rem}\label{relation-hyp-func}
 For simplicity, we will call a transnormal function $f$
a \emph{0-isoparametric function}, denoted by
$f\in\mathscr{I}_0(N^{n+1})$ (and by $\triangle_0f:=|\nabla f|^2$).
Note that a $1$-isoparametric function is exactly the usual
isoparametric function introduced at the beginning of this paper.
Generally, in a Riemannian manifold $N^{n+1}$, a $k$-isoparametric
hypersurface can not determine a corresponding global
$k$-isoparametric function. However, as we stated before, a
$1$-isoparametric hypersurface in a sphere does determine a
corresponding global isoparametric function according to
Cartan-M\"{u}nzner's construction of isoparametric polynomial.
Furthermore, in a compact symmetric space $N^{n+1}$, a
$1$-isoparametric hypersurface $M^n$ also determines a corresponding
global isoparametric function $f$ on $N^{n+1}$. To show this
assertion, first we know that by \cite{HLO} $M^n$ must be an
equifocal hypersurface ( cf.\cite{TT95}, \cite{Tang}). Next, it
follows from Terng and Thorbergsson \cite{TT95} that $M^n$
determines a transnormal system on $N^{n+1}$ with t-regular foils of
codimension one, which then by Miyaoka \cite{Miy13} corresponds to a
global transnormal function $\bar{f}$ on $N^{n+1}$ whose regular
level hypersurfaces are parallel to $M^n$ and have constant mean
curvatures. Finally we get a desired global isoparametric function
$f$ on $N^{n+1}$ via $\bar{f}$ with the same level sets by some
regularization.
\end{rem}


It follows directly from the definition that the sets of
$1$-,$2$-,$\cdots$,$n$-isoparametric functions (hypersurfaces) induce a
filtration of isoparametric functions (hypersurfaces) on a
Riemannian manifold $N^{n+1}$ as:
\begin{equation}\label{isop sequence}
\Big(\mathscr{I}_0(N^{n+1})\supset\Big)\mathscr{I}_1(N^{n+1})\supset\cdots\supset\mathscr{I}_n(N^{n+1}).
\end{equation}

By a straightforward verification, we will see in the next section
that, a hypersurface $M^n$ in $N^{n+1}$ is $k$-isoparametric if and
only if all its nearby parallel hypersurfaces, say $M_{t}$
($M_{t_0}$=$M$), have constant $i$-th mean curvatures $H_i(t)$ for
$i=1,\cdots,k$, where $H_i(t):=\sigma_i(S(t))=\sigma_i(\mu(t))$ is
the $i$-th elementary symmetric polynomial of the shape operator
$S(t)$ or principal curvatures $\mu(t)=(\mu_1(t),\cdots,\mu_n(t))$
of $M_t$. In particular, $M^n$ is an $n$-isoparametric hypersurface
if and only if all its nearby parallel hypersurfaces have constant
principal curvatures. In this case, $M^{n}$ is called a
\emph{totally isoparametric hypersurface} and the corresponding
(local) function a \emph{totally isoparametric function}. In this
way, Cartan's rigidity result can be restated as:
\begin{equation}\label{cartan result}
\emph{A 1-isoparametric hypersurface in a real space form is totally
isoparametric.}
\end{equation}

Although Cartan's rigidity result can hardly hold in a general
Riemannian manifold, we will be able to extend it
to those Riemannian manifolds with some symmetries other than real space forms, as stated in the following theorems.


\begin{thm}\label{existence in CPn}
Let $\mathbb{C}P^{m}$ be the complex projective space equipped with
Fubini-Study metric of constant holomorphic sectional curvature $4$.
Then
\begin{itemize}
\item[(i)] A $1$-isoparametric hypersurface in a complex
even-dimensional projective space $\mathbb{C}P^{2n}$ is totally
isoparametric. In fact, it is homogeneous.
\item[(ii)] Each complex odd-dimensional projective space
$\mathbb{C}P^{2n+1}$ admits a $2$-isoparametric hypersurface which is not $3$-isoparametric.
\end{itemize}
\end{thm}


The homogeneity conclusion in (i) follows from the following sequence
of equivalent conditions for an \emph{isoparametric hypersurface}
$\widetilde{M}^{2n-1}$ in $\mathbb{C}P^{n}$ by putting results of
\cite{Wa82}, \cite{Kim}, \cite{Pa}\footnote{It was pointed out by
\cite{Xiao} that there exist some mistakes in \cite{Pa}.
However, the conclusions we cited are
correct.} and \cite{Xiao} together:
\begin{eqnarray}\label{equiv-CPC-homog}
&&\widetilde{M}~ \emph{has constant principal
curvatures}\\
&\Leftrightarrow&~\widetilde{M}~ \emph{is
Hopf, i.e.,}~ J\tilde{\nu} ~\emph{is
principal}\nonumber\\
&\Leftrightarrow&\emph{one of the focal
submanifolds is complex}\nonumber\\
&\Leftrightarrow&
\widetilde{M}~ \emph{is homogeneous (i.e. an open part of a
homogeneous hypersurface)}\nonumber\\
&\Leftrightarrow&l\equiv2\Leftrightarrow
l\equiv const\Leftrightarrow \tilde{g}\equiv
const\nonumber\\
&\Leftrightarrow&\widetilde{M}~ \emph{is totally
isoparametric,}\nonumber
\end{eqnarray}
where $\tilde{\nu}$ is a unit normal vector field on
$\widetilde{M}$, $J$ the canonical complex structure of
$\mathbb{C}P^n$, $\tilde{g}$ the number of distinct principal
curvatures of $\widetilde{M}$ and $l$ the number of non-horizontal
eigenspaces of the shape operator on $M:=\pi^{-1}(\widetilde{M})$ by
the Hopf fibration  $\pi: S^{2n+1}\rightarrow \mathbb{C}P^n$.

It is worth to point out that hypersurfaces with constant principal
curvatures in any Riemannian manifold other than a real space form
are far from being classified; and even the set of $g$, the number
of distinct principal curvatures, has not been determined so well as
M\"{u}nzner did for such hypersurfaces in spheres (see \cite{Bern10}
for a detailed survey). Combining with the remarkable classification
of homogeneous hypersurfaces in complex projective spaces by Takagi
\cite{Tak}, Theorem $1.1$ (i) classifies completely isoparametric
hypersurfaces in $\mathbb{C}P^{2n}$ indeed. Our classification
should be compared with the case of complex hyperbolic space
$\mathbb{C}H^n$ in which \cite{DD10} recently constructed
inhomogeneous examples of isoparametric hypersurfaces for each
$n\geq3$.

Examples in Theorem $1.1$(ii) are constructed explicitly by projecting certain OT-FKM-type
isoparametric hypersurfaces in spheres by the Hopf fibration. Here,
the deduction that $1$-isoparametric is sufficient for $2$-isoparametric
in $\mathbb{C}P^n$ (or more generally in an Einstein manifold) can
be easily seen from relations of the shape operators of $M$ and
$\widetilde{M}$ by the Hopf fibration (or from the Riccati
equation).

A deeper exploration of these relations by the Hopf fibration
and the equivalence sequence (\ref{equiv-CPC-homog}) will lead us to a
classification of isoparametric hypersurfaces in $\mathbb{C}P^n$
with constant 3rd mean curvatures as follows:


\begin{thm}\label{H3classification in Cpn}
A $1$-isoparametric hypersurface in $\mathbb{C}P^n$ with constant
3rd mean curvature, \emph{i.e.,} $H_3\equiv const$, is totally
isoparametric and hence homogeneous.
\end{thm}
We expect this theorem to play a special role in solving Chern
conjecture, which asserts that a closed hypersurface in a sphere
with constant 1st and 2nd mean curvatures must be an isoparametric
hypersurface (cf. \cite{SW08}, \cite{GT5}). This conjecture has been
proved only for the case of $3$-dimensional closed hypersurfaces in
$S^{4}$, while remains open for higher dimensional cases. On the
other hand, one has not any example of inhomogeneous hypersurface
with constant principal curvatures in Riemannian symmetric spaces
other than real space forms (cf. \cite{Bern10}). It turns out that
there are some relations between these two questions as stated in
the following.


\begin{cor}\label{Chern exam}
Suppose that $\widetilde{M}$ is an inhomogeneous hypersurface in
$\mathbb{C}P^n$ with constant mean curvatures
$H_1,H_2,H_3$. Then the inverse image $M=\pi^{-1}(\widetilde{M})$ in
$S^{2n+1}$ under the Hopf fibration is a non-isoparametric
hypersurface with constant first mean curvature $H_1$ and constant
second mean curvature $H_2-1$, giving a counterexample
to Chern conjecture.
\end{cor}

In general, the ambient space $N$ is lack of such ``satisfied
structures" (\emph{e.g.}, Hopf fibration, complex structure,
explicit representation of curvature tenser, \emph{etc.}) as
$\mathbb{C}P^{n}$, resulting in obstructions for us to get rigidity
results as Theorems \ref{existence in CPn}, \ref{H3classification in
Cpn} for $CP^{n}$. However, when $N$ is a complex space form or more
generally a locally rank one symmetric space, there still exist
certain symmetries of the curvature tensor, which make the Riccati
equation more useful in dealing with parallel hypersurfaces in such
spaces than in general Riemannian manifolds as in \cite{GT10}. For
example, by making use of the Riccati equation, we obtain the
following rigidity result (compared with Theorem
\ref{H3classification in Cpn} where $H_3\equiv const$ is an
assumption weaker than $3$-isoparametric):


\begin{thm}\label{3isop rk one sp}
A $3$-isoparametric hypersurface $M^n$ in a locally rank one
symmetric space $N^{n+1}$ is $5$-isoparametric. If in addition
$N^{n+1}$ is locally a complex space form, then $M^n$ is totally
isoparametric.
\end{thm}


\begin{rem}\label{Jacobi remark}
The key point in the proof of this theorem is that the normal Jacobi
operator $K_{\nu}:\mathcal {T}M\rightarrow\mathcal {T}M$ defined by
$K_{\nu}(X):=R(\nu,X)\nu=(\nabla_{[\nu,X]}-[\nabla_\nu,\nabla_X])\nu$,
for $X\in\mathcal {T}M$, where $\nu$ is a unit normal vector field
on $M$, has constant eigenvalues and is parallel along the normal
geodesics. In fact, if both $tr(K_{\nu})$ and $tr(K_{\nu}^2)$ are
constant, in the same way we find that $3$-isoparametric is
sufficient for $4$-isoparametric in a locally symmetric space.
Fortunately, there are many locally symmetric spaces with constant
$tr(K_{\xi})$ and constant $tr(K_{\xi}^2)$, independent of the
choice of the unit tangent vector $\xi$. Such locally symmetric
spaces are involved in the Lichnerowicz conjecture and have been
classified in \cite{CGW}.
\end{rem}

The following rigidity result is another application of the
Riccati equation. To state it, we need to introduce the concept of
\emph{curvature-adapted} or \emph{compatible} hypersurfaces (resp.
submanifolds), namely, whose normal Jacobi operator
$K_{\nu}$ and shape operator $S_{\nu}$ (resp. $S_{\nu}\oplus I$)
commute, or equivalently, are simultaneously diagonalizable for each
unit normal vector $\nu$ (cf. \cite{Bern91}, \cite{Gr04}).


\begin{thm}\label{compatible rk one sp}
Let $M^n$ be a curvature-adapted hypersurface in a locally rank one
symmetric space $N^{n+1}$. If either
\begin{itemize}
\item[(i)] $M^n$ has constant principal curvatures, or
\item[(ii)]$M^n$ is a $1$-isoparametric hypersurface,
\end{itemize}
then $M^n$ is totally isoparametric.
\end{thm}


\begin{rem}
The proof of Theorem$1.4$(i) yields also that a tube (tubular hypersurface) $M^n$ around a
curvature-adapted submanifold of constant principal curvatures in a
locally rank one symmetric space $N^{n+1}$ is a curvature-adapted
hypersurface of constant principal curvatures and thus totally
isoparametric.
\end{rem}


\begin{rem}\label{Osserman}
It is clear to see that(cf. \cite{Gr04}), given a curvature-adapted
hypersurface $M$ in a locally symmetric space $N$, each nearby
parallel hypersurface $M_{t}$ is automatically curvature-adapted.
Theorem $1.4$ holds also for a hypersurface $M$ in an Osserman
manifold whose nearby parallel hypersurfaces $M_{t}$ are all
curvature-adapted. In fact, the proof of Theorem $1.4$ depends
mainly on the constancy of eigenvalues of the Jacobi operator, while
an Osserman manifold is exactly a Riemannian manifold $N$ whose
Jacobi operator has constant eigenvalues including multiplicities,
independent of the choice of the unit tangent vector and the point
on $N$. Essentially, Osserman conjectured that an Osserman manifold
( named later ) is a locally rank one symmetric space. This
conjecture has been verified to be true except for the case when
dim$N$=16 (cf. \cite{Chi1}, \cite{N1}, \cite{BGN}).
\end{rem}

In a locally rank one symmetric space $N^{n+1}$ with non-constant
sectional curvatures, all known examples of totally isoparametric
hypersurfaces are homogeneous. Recall that for a hypersurface in $\mathbb{C}P^n$,
totally isoparametric is equivalent to homogeneous by the equivalence sequence (\ref{equiv-CPC-homog}).
In all probability, this equivalence still holds, at least, in each compact case ($\mathbb{C}P^n$,
$\mathbb{H}P^n$, $\mathbb{O}P^2$). On the other hand,
curvature-adapted hypersufaces in complex space forms
($\mathbb{C}P^n$, $\mathbb{C}H^n$, $\mathbb{C}^n$) are just Hopf
hypersurfaces. Similar to Kimura's work in $\mathbb{C}P^n$ (cf.
\cite{Kim}), Berndt \cite{Bern89} proved that a Hopf hypersurface
with constant principal curvatures in $\mathbb{C}H^n$ is necessarily
homogeneous. Based on this remarkable result, Theorem
\ref{compatible rk one sp} (ii) yields


\begin{cor}
A $1$-isoparametric Hopf hypersurface in $\mathbb{C}H^n$ is
homogeneous.
\end{cor}

We conclude this section with some remarks. In virtue of the
classification of homogeneous hypersurfaces in $\mathbb{C}H^n$ by
\cite{BT}, $1$-isoparametric Hopf hypersurfaces in $\mathbb{C}H^n$
are consequently classified. An interesting phenomenon appeared in
$\mathbb{C}H^n$ that there exist many non-Hopf homogeneous
hypersurfaces, which is quite different from that in $\mathbb{C}P^n$
(cf. \cite{Bern10}). As is well known, a hypersurface in a non-flat
complex space form is curvature-adapted if and only if it is Hopf.
The concept of curvature-adapted hypersurface gives a natural
generalization of Hopf hypersurfaces in Hermitian manifolds to more
general Riemannian manifolds. Surprisingly, Berndt (\cite{Bern91})
proved that a hypersurface in quaternionic projective space
$\mathbb{H}P^n$ is curvature-adapted if and only if it is
homogeneous. While in quaternionic hyperbolic space $\mathbb{H}H^n$,
the classification of curvature-adapted hypersurfaces is still an
open problem (cf. \cite{Mur10}, \cite{Bern91}). Moreover, the
classification of curvature-adapted hypersufaces in octonionic space
forms is still elusive.

\section{Hidden Cartan-M\"{u}nzner equations}\label{section-CM}
This section will be devoted to the establishment of an inductive formula for those $(n-1)$
equations implied by Cartan-M\"{u}nzner equations
(\ref{eq1.3})-(\ref{eq1.4}) for isoparametric functions (polynomials)
on $S^{n+1}$. We first show the following
geometric characterization of a $k$-isoparametric hypersurface $M^n$
defined by a (local) $k$-isoparametric function $f$ on a Riemannian
manifold $N^{n+1}$:


\begin{lem}\label{geom. meaning for kisop}
A hypersurface is $k$-isoparametric if and only if each of its
nearby parallel hypersurfaces has constant $i$-th mean curvatures
for $i=1,\cdots,k$.
\end{lem}

\begin{proof}
Let $M^n:=f^{-1}(t_0)$ be a $k$-isoparametric hypersurface in a
Riemannian manifold $N^{n+1}$, where $t_{0}$ is a regular value of the (local)
$k$-isoparametric function $f$ satisfying equations (\ref{iso1}) and
(\ref{k-isop-a}). For $\varepsilon>0$ sufficiently small, $t\in
(t_0-\varepsilon, t_0+\varepsilon)$, $M_t:=f^{-1}(t)$ is still a
hypersurface that is parallel to $M$ by equation (\ref{iso1}), since
now $\nabla f/|\nabla f|$ is the tangent vector field along the
normal geodesic of $M$ at each point. It is well known that the
shape operator, say $S(t)$, of $M_t$ with respect to the unit normal
vector field $\nu=\nabla f/|\nabla f|$ is characterized by (cf.
\cite{CR85}):
\begin{equation}\label{shape-hessian}
\langle S(t)X,~Y\rangle=\frac{-H_f(X,Y)}{|\nabla f|},
\end{equation}
where $X,Y$ are tangent vectors to $M_t$ and $H_f$ the Hessian of
$f$. Now let $\{e_1,\cdots,e_n\}$ be a local orthonormal basis of
$M_t$ and $e_i$ the eigenvector of $S(t)$ with respect to
principal curvature $\mu_i$ for $i=1,\cdots,n$. Then it follows from equation
(\ref{iso1}) that $H_f(e_i, \nu)=0$ and
$H_f(\nu,\nu)=b'(f)/2$. Thus under the orthonormal frame
$\{e_1,\cdots,e_n,\nu\}$ of $N^{n+1}$, the Hessian $H_f$ is
expressed as the diagonal matrix
\begin{equation}\label{Hf S}
H_f=diag\Big(-\sqrt{b(f)}\mu_1,\cdots,-\sqrt{b(f)}\mu_n,
b'(f)/2\Big).
\end{equation}
Therefore, a straightforward calculation using (\ref{Hf S}) shows
that $\triangle_jf:=\sigma_j(H_f)$ can be expressed in terms of the
mean curvatures $H_i:=H_i(t)=\sigma_i(S(t))$ for $i\leq j$ as:
\begin{equation}\label{lapk Hk}
\triangle_jf=\Big(-\sqrt{b(f)}\Big)^jH_j+\Big(-\sqrt{b(f)}\Big)^{j-1}\frac{b'(f)}{2}H_{j-1},
\end{equation}
and conversely,
\begin{equation}\label{Hk lapk}
H_j=\frac{1}{(2\sqrt{b(f)})^j}\Big(\sum_{i=1}^{j}(-1)^i2^i(b'(f))^{j-i}\triangle_if+(b'(f))^j\Big).
\end{equation}
Preparing these equations, we are now in a position to complete the
proof of the lemma. First, for a given $k$-isoparametric function
$f$, the equations (\ref{k-isop-a}) and (\ref{Hk lapk}) yield that
$H_1,\cdots,H_k$ are functions of $f$, thus constant on each $M_t$.
Consequently, the nearby parallel hypersurfaces of a
$k$-isoparametric hypersurface have constant mean curvatures
$H_1,\cdots,H_k$. Conversely, if each nearby parallel hypersurface
$M'_t:=exp_M(t\nu)$ of $M^n$ (the image of $M^n$ under the normal
exponential map at distance $t\in(-\varepsilon, \varepsilon)$) has
constant mean curvatures $H_1,\cdots,H_k$ (continuously depend on
$t$), we can define a function $f$ on the local neighborhood
$\bigcup_{t\in(-\varepsilon, \varepsilon)}M'_t\subset N^{n+1}$ of
$M^n$ in $N^{n+1}$ by $f|_{M'_t}:= t$ for $t\in(-\varepsilon,
\varepsilon)$. Clearly $|\nabla f|^2=1$ and by (\ref{lapk Hk}),
$\triangle_1f,\cdots,\triangle_kf$ are constant on each $M'_t$ and
continuously depend on $t=f$. Namely, $f$ is a local
$k$-isoparametric function on $N^{n+1}$ and thus $M^n$ is
$k$-isoparametric.
\end{proof}


In particular, an $n$-isoparametric hypersurface $M^n$ has constant mean
curvatures $H_1,\cdots,H_n$, thus constant principal curvatures. This justifies
the notion of \emph{totally isoparametric}. As mentioned in the introduction, by Cartan's
rigidity result (\ref{cartan result}), we know that for an
isoparametric function $f=F|_{S^{n+1}}$ satisfying (\ref{CM-isop-f}) on $S^{n+1}$, the restriction of a
Cartan polynomial $F$ satisfying (\ref{eq1.3}) and (\ref{eq1.4}), the mean
curvatures $H_1,\cdots,H_n$, or equivalently
$Q_1:=\rho_1(S(t)),\cdots,Q_n:=\rho_n(S(t))$, where $S(t)$ is the
shape operator of the level hypersurface $M_t:=f^{-1}(t)$, are
constant on $M_t$ and continuously (smoothly, in fact) depend on
$f=t\in(-1,1)$. This argument together with (\ref{lapk Hk}) allows us to construct $(n-1)$ smooth functions $a_2,\cdots,a_n\in
C^{\infty}(\mathbb{R})$ other than the functions $b$ and $a$ in $(\ref{CM-isop-f})$ by
\begin{equation}\label{ai for isop in Sn}
\triangle_2f=a_2(f),\cdots,\triangle_nf=a_n(f);
\end{equation}
or equivalently, $(n-1)$ smooth functions
$p_2,\cdots,p_n\in C^{\infty}(\mathbb{R})$ $(p_1=a_1=a)$ by
\begin{equation}\label{pi for isop in Sn}
\rho_2(H_f)=p_2(f),\cdots,\rho_n(H_f)=p_n(f).
\end{equation}
Correspondingly, we find out the $(n-1)$ \emph{hidden
Cartan-M\"{u}nzner equations} for polynomial $F$ involving
$\triangle_iF:=\sigma_i(H_F)$, or equivalently, involving
$\rho_i(H_F)$. So once we have formulae for one of the sets
$\{H_i\}$, $\{Q_i\}$, $\{a_i\}$, $\{p_i\}$,
$\{\sigma_i(H_F)|_{S^{n+1}}=:\bar{\sigma}_i\}$, and
$\{\rho_i(H_F)|_{S^{n+1}}=:\bar{\rho}_i\}$, the others can be
obtained by Newton's identities (\ref{Newton identity}), the
equalities (\ref{lapk Hk}), (\ref{Hk lapk}), and the relation
between the Hessian $H_f$ of $f$ on $S^{n+1}$ and the Hessian $H_F$
of $F$ on $\mathbb{R}^{n+2}$. In this way, the M\"{u}nzner's
geometric construction of $f$ and the formulae for principal
curvatures of $M_t$ lead us to an inductive formula for the set
$\{Q_i\}$, and then for the set $\{\bar{\rho}_i\}$ as follows
($Q_i$, $\bar{\rho}_i$ are regarded as functions of $t=f\in(-1,1)$):


\begin{thm}\label{implicit CM eqs}
In the same notations as above, for $k=1,\cdots,n-1$, the following
equalities are valid.
\begin{eqnarray}
&&Q_{k+1}=\frac{g}{k}\sqrt{1-t^2}\frac{dQ_k}{dt}-Q_{k-1},\nonumber\\
&&Q_0=n,\;Q_1=\frac{m_1g}{2}\sqrt{\frac{1+t}{1-t}}-\frac{m_2g}{2}\sqrt{\frac{1-t}{1+t}}.\nonumber\\
&&\bar{\rho}_{k+1}=\begin{cases}-\frac{g^2}{k}(1-t^2)\frac{d\bar{\rho}_{k}}{dt}
-g(g-2)t\bar{\rho}_{k}+g^2(g-1)\bar{\rho}_{k-1}\nonumber\\
+2g^{k+1}(g-1)^k(g-2),\qquad\qquad\quad for~~k~odd;\\
-\frac{g^2}{k}(1-t^2)\frac{d\bar{\rho}_{k}}{dt}
-g(g-2)t\bar{\rho}_{k}+g^2(g-1)\bar{\rho}_{k-1}\nonumber\\
+2g^{k+1}(g-1)^k(g-2)t,\qquad\quad\quad\ for~~k~even,
\end{cases}\nonumber\\
&&\bar{\rho}_{0}=n+2,\quad
\bar{\rho}_{1}=\frac{g^2}{2}(m_2-m_1).\nonumber
\end{eqnarray}
\end{thm}


\begin{rem}
Since $F$ is a homogeneous polynomial of degree $g$ on
$\mathbb{R}^{n+2}$, we could homogenize the expressions of
$\bar{\rho}_{k}$ so as to extend the Cartan-M\"{u}nzner
equations (\ref{eq1.3}), (\ref{eq1.4}) on $\mathbb{R}^{n+2}$. For
instance, by the inductive formula, we list $(n\geq4)$:
\begin{eqnarray}
\rho_2(H_F)&=&-\frac{g^3}{2}(g-2)(m_2-m_1)F|x|^{g-4}+g^2(g-1)(n+2g-2)|x|^{2g-4},\nonumber\\
\rho_3(H_F)&=&\frac{g^4}{4}(g-2)(g-4)(m_2-m_1)F^2|x|^{g-6}-ng^3(g-1)(g-2)F|x|^{2g-6}\nonumber\\
&&+\frac{g^4}{4}(g^2-2)(m_2-m_1)|x|^{3g-6},\nonumber\\
\rho_4(H_F)&=&-\frac{g^5}{12}(g-2)(g-4)(g-6)(m_2-m_1)F^3|x|^{g-8}\nonumber\\
&&+\frac{2n}{3}g^4(g-1)(g-2)(g-3)F^2|x|^{2g-8}\nonumber\\
&&-\frac{g^5}{12}(g-2)(5g^2-2g-12)(m_2-m_1)F|x|^{3g-8}\nonumber\\
&&+\Big(\frac{n}{3}g^4(g-1)(g^2+g-3)+2g^4(g-1)^4\Big)|x|^{4g-8}.\nonumber
\end{eqnarray}
\end{rem}
\begin{proof}
According to M\"{u}nzner \cite{Mu80}, the level
hypersurface $M_t^n:=f^{-1}(t)$ $(t\in(-1,1))$ of $f$ has $g$
distinct principal curvatures
$\{\lambda_i=\cot(\tau+\frac{(i-1)\pi}{g})|i=1,\cdots,g\}$ with
multiplicities $m_i$ satisfying $m_i=m_{i+2}$ (subscripts $mod$ g),
where $f=t=\cos(g\tau)$ on $M_t$ and $\tau\in(0,\frac{\pi}{g})$ is
in fact the oriented (with respect to the unit normal vector field
$\nu:=\nabla f/|\nabla f|$) distance from $M_t$ to the focal
submanifold $M_{+}:=f^{-1}(1)$ (cf. \cite{CR85}). As a consequent result we have
\begin{equation}\label{Qk}
Q_k=m_1\sum_{i=1}^{[\frac{g+1}{2}]}\Big(\cot(\tau+\frac{2(i-1)\pi}{g})\Big)^k+
m_2\sum_{i=1}^{[\frac{g}{2}]}\Big(\cot(\tau+\frac{(2i-1)\pi}{g})\Big)^k.
\end{equation}
Since it is difficult to give a general formula for high order
power sum of the cotangent functions, we turn to give an inductive
formula instead of a general formula for $Q_k$. Observe that
\begin{eqnarray*}
&&\frac{d}{d\tau}(\cot(\tau+\theta))^k=-k\Big((\cot(\tau+\theta))^{k-1}+(\cot(\tau+\theta))^{k+1}\Big),\nonumber\\
&&\frac{dt}{d\tau}=-g\sin(g\tau)=-g\sqrt{1-t^2},\nonumber
\end{eqnarray*} and thus
\begin{equation*}
(\cot(\tau+\theta))^{k-1}+(\cot(\tau+\theta))^{k+1}=\frac{g}{k}\sqrt{1-t^2}\frac{d}{dt}(\cot(\tau+\theta))^k
\end{equation*}
which implies immediately the first inductive formula of the theorem by
taking sum in (\ref{Qk}).

To distinguish the notations, we denote by $\nabla$ and $D$ the
Levi-Civita connections on $S^{n+1}$ and $\mathbb{R}^{n+2}$,
respectively. By definition, we have for $X,Y\in\mathcal
{T}S^{n+1}$,
\begin{eqnarray}\label{Hf HF}
&&H_F(X,Y)=X(YF)-(D_XY)F=X(YF)-(\nabla_XY)F+\langle X,Y\rangle \frac{\partial F}{\partial r}\\
&&\quad\ \ \quad\qquad=H_f(X,Y)+\langle X,Y\rangle gf\nonumber,
\end{eqnarray}
where $\frac{\partial F}{\partial r}$ is the partial derivative of
$F$ with respect to the radial direction.

Let $\{e_1,\cdots,e_n\}$ be the principal orthonormal frame of $M_t$
as in the proof of Lemma \ref{geom. meaning for kisop} and be
arranged such that under this frame the shape operator
$$S(t)=diag(\mu_1,\cdots,\mu_n)=diag(\lambda_1I_{m_1},\cdots,\lambda_gI_{m_g}).$$
Then $\{e_1,\cdots,e_n, \nu_x:=\frac{\nabla f(x)}{|\nabla f(x)|}\}$
is an orthonormal frame of $\mathcal {T}_xS^{n+1}$ and under this
frame the Hessian $H_f$ can be expressed as in formula (\ref{Hf S}),
while $\{e_1,\cdots,e_n, \nu_x, x\}$ is an orthonormal frame of
$\mathbb{R}^{n+2}$ at $x\in S^{n+1}$. It is easily seen that
\begin{equation*}
H_F(e_i,x)=0,\quad H_F(\nu_x,x)=(g-1)|\nabla f|,\quad H_F(x,x)=g(g-1)f,
\end{equation*}
and thus by (\ref{Hf S}), (\ref{Hf HF}), the Hessian $H_F$ at $x\in
S^{n+1}$ can be expressed as
\begin{equation}\label{HF}
H_F=\left(\begin{smallmatrix} -\sqrt{b(f)}\mu_1+gf&&&&\\
&\ddots&&&\\&&-\sqrt{b(f)}\mu_n+gf&&\\
&&&\frac{b'(f)}{2}+gf&(g-1)\sqrt{b(f)}\\
&&&(g-1)\sqrt{b(f)}&g(g-1)f
\end{smallmatrix}\right)
\end{equation}
which has eigenvalues
$\{-g\sqrt{1-f^2}\mu_1+gf,\cdots,-g\sqrt{1-f^2}\mu_n+gf,g(g-1),-g(g-1)\}$
by using formula (\ref{CM-isop-f}). Consequently, we have on $M_t=f^{-1}(t)$,
\begin{eqnarray}\label{phokbar}
&&\bar{\rho}_{k}=m_1\sum_{i=1}^{[\frac{g+1}{2}]}\Big(-g\sqrt{1-t^2}\cot(\tau+\frac{2(i-1)\pi}{g})+gt\Big)^k\\
&&\ \quad\ +m_2\sum_{i=1}^{[\frac{g}{2}]}\Big(-g\sqrt{1-t^2}\cot(\tau+\frac{(2i-1)\pi}{g})+gt\Big)^k\nonumber\\
&&\ \quad\ +g^k(g-1)^k(1+(-1)^k).\nonumber
\end{eqnarray}
At last, by taking derivative of $\bar{\rho}_{k}$ with respect to
$t$ in (\ref{phokbar}) and using the relation $t=\cos(g\tau)$, we
arrive at the second inductive formula of the theorem.
\end{proof}


\textbf{\emph{Example.}} As is well known, Cartan's polynomial
$F:\mathbb{R}^{3m+2}\longrightarrow \mathbb{R}$ for isoparametric
hypersurfaces in spheres with $g=3$ distinct principal curvatures
can be written as
\begin{equation*}
F(x)=u^{3}-3uv^{2}+\frac{3}{2}u(X\overline{X}+Y\overline{Y}-2Z\overline{Z})+\frac{3\sqrt{3}}{2}v(X\overline{X}-Y\overline{Y})+\frac{3\sqrt{3}}{2}
(XYZ+\overline{XYZ}).
\end{equation*}
In this formula, $x=(u,v,X,Y,Z)\in\mathbb{R}^{3m+2}$, $u$ and $v$ are real parameters, while $X$, $Y$, $Z$ are coordinates in the algebra $F=\mathbb{R},\mathbb{C},\mathbb{H}$(Quaternions) ~or~ $\mathbb{O}$(Cayley numbers), for the case
$m_1=m_2=m=1,2,4, ~or~ 8$, respectively. For example, we compute the Hessian of
$F$ in the case of $m=1$:
\begin{equation*}
H_F = 3
\begin{pmatrix}
2u & -2v & X & Y & -2Z\\
-2v& -2u & \sqrt{3}X & -\sqrt{3}Y & 0\\
X  & \sqrt{3}X & u+\sqrt{3}v & \sqrt{3}Z & \sqrt{3}Y \\
Y  & -\sqrt{3}Y & \sqrt{3}Z  & u-\sqrt{3}v & \sqrt{3}X \\
-2Z & 0 & \sqrt{3}Y & \sqrt{3}X & -2u
\end{pmatrix}.
\end{equation*}
Then direct calculations lead to
\begin{eqnarray*}
&&\triangle_{1}(F)=0,\quad \triangle_{2}(F)=-63|x|^2,\quad\triangle_{3}(F)=-54F,\\
&&\triangle_{4}(F)=3^5\cdot 4|x|^4,\quad\triangle_{5}(F)=2^3\cdot 3^5|x|^2F.\nonumber
\end{eqnarray*}

\section{Isoparametric hypersurfaces in complex projective spaces}
In this section we begin by establishing an equivalence condition
for an isoparametric hypersuface $\widetilde{M}^{2n-1}$ in
$\mathbb{C}P^n$ to have constant $3$rd mean curvature $H_3$, that is
the constancy of an $S^1$-invariant function, say $\alpha$, on
$M^{2n}:=\pi^{-1}(\widetilde{M}^{2n-1})\subset S^{2n+1}$ ($\pi$ is
the Hopf fibration). As a consequence, $\widetilde{M}^{2n-1}$ is
$3$-isoparametric if and only if $\alpha$ is constant on each nearby
parallel hypersurface $M_t$ of $M^{2n}$. Next, we turn to prove
Theorem \ref{existence in CPn} and Theorem \ref{H3classification in
Cpn}. In particular, we construct explicitly some $S^1$-invariant
OT-FKM-type isoparametric polynomials on $\mathbb{R}^{4n+4}$,
calculating the function $\alpha$ which turns out non-constant on
some level hypersurface. In this way, we get finally the examples in
(ii) of Theorem \ref{existence in CPn}, as desired.

Let $\widetilde{M}^{2n-1}$ be a hypersurface in $\mathbb{C}P^n$ with
unit normal vector field $\tilde{\nu}$. Observe that the unit normal vector
field $\nu$ of $M^{2n}:=\pi^{-1}(\widetilde{M}^{2n-1})\subset
S^{2n+1}$ ($\pi$ is the Hopf fibration) is just the horizontal lift
of $\tilde{\nu}$, \emph{i.e.}, $\pi_{*}\nu=\tilde{\nu}$. For simplicity, we will use
the same symbols for Levi-Civita connections, shape
operators, Hessians, \emph{etc}, on spheres and Euclidean spaces as
last section and only add a tilde to the corresponding symbols on
$\mathbb{C}P^n$. Let $\tilde{J}$ be the complex structure on
$\mathbb{C}P^n$ induced from the canonical complex structure $J$ on
$\mathbb{R}^{2n+2}$ by the Hopf fibration, so that $Jx$ is the
tangent vector field of the $S^1$-fibre through $x\in S^{2n+1}$. It
follows that $J\nu$ is the horizontal lift of $\tilde{J}\tilde{\nu}$,
thus a global unit tangent vector field of $M^{2n}$ perpendicular with
$Jx$. These arguments allow us to choose a local orthonormal basis
$\{\tilde{e}_1,\cdots,\tilde{e}_{2n-2},\tilde{J}\tilde{\nu}\}$ on
$\widetilde{M}^{2n-1}$, so that the shape operator
$\widetilde{S}_{\tilde{\nu}}$ of $\widetilde{M}^{2n-1}$ is expressed
by a symmetric matrix $\widetilde{S}$. Following \cite{Wa82}, the shape
operator $S_{\nu}$ of $M^{2n}$ under the orthonormal
basis $\{e_1,\cdots,e_{2n-2},J\nu,Jx\}$ $(\pi_{*}e_i=\tilde{e}_i)$
can be expressed by a symmetric matrix $S$:
\begin{equation}\label{shapeS and Stilde in CPn}
S=\left(\begin{array}{ccc|c} &&&0\\
&\widetilde{S}& & \vdots \\
&&&-1\\
 \hline 0&\cdots&-1 & 0
\end{array}\right)
\end{equation}
To see the relation of $S$ and $\widetilde{S}$ above, we remark that the Hopf fibration is a
Riemannian submersion with totally geodesic $S^1$-fibres, and it follows that
$[Jx, \nu]=0$, and
$\langle
\widetilde{S}_{\tilde{\nu}}(\widetilde{X}),\widetilde{Y}\rangle=\langle
S_{\nu}(X),Y\rangle$ for $\widetilde{X}, \widetilde{Y}\in\mathcal
{T}\widetilde{M}^{2n-1}$ and their
horizontal lift $X,Y\in\mathcal {T}M^{2n}$. Hence
\begin{equation}\label{SJx}
-S_{\nu}(Jx)=\nabla_{Jx}\nu=\nabla_{\nu}Jx=JD_{\nu}x=J\nu.
\end{equation}
 Now
define the $S^1$-invariant function $\alpha$ on $M^{2n}$ ( it will
play an important role in this section ) by
\begin{equation}\label{alpha function}
\alpha:=\langle S_{\nu}(J\nu),J\nu\rangle=\langle
\nu,\nabla_{J\nu}J\nu\rangle=\langle\widetilde{S}_{\tilde{\nu}}(\tilde{J}\tilde{\nu}),\tilde{J}\tilde{\nu}\rangle\circ
\pi.
\end{equation}
Therefore, using (\ref{shapeS and Stilde in CPn}), we derive that
\begin{eqnarray}\label{sigma123ofS Stilde}
&&\sigma_1(S_{\nu})=\sigma_1(\widetilde{S}_{\tilde{\nu}})\circ\pi,\ \sigma_2(S_{\nu})=\sigma_2(\widetilde{S}_{\tilde{\nu}})\circ\pi-1,\\
&&\sigma_3(S_{\nu})=\sigma_3(\widetilde{S}_{\tilde{\nu}})\circ\pi-(\sigma_1(\widetilde{S}_{\tilde{\nu}})\circ\pi-\alpha).\nonumber
\end{eqnarray}
Note that the inverseimages under the Hopf fibration of
parallel hypersurfaces in $\mathbb{C}P^n$ are still parallel hypersurfaces in
$S^{2n+1}$. Now by (\ref{sigma123ofS Stilde}) we obtain the following.


\begin{prop}(cf. \cite{Wa82})\label{lift of isop to Sn}
A hypersurface $\widetilde{M}^{2n-1}$ in $\mathbb{C}P^n$ is
isoparametric if and only if its inverse image
$M^{2n}:=\pi^{-1}(\widetilde{M}^{2n-1})$ under the Hopf fibration
$\pi$ is an isoparametric hypersurface in $S^{2n+1}$.\hfill $\Box$
\end{prop}


\begin{cor}\label{3isop-alpha}
Let $\widetilde{M}^{2n-1}$ be a $1$-isoparametric hypersurface in
$\mathbb{C}P^n$. Then
\begin{itemize}
 \item[a)] It must be $2$-isoparametric;
 \item[b)] It has constant $3$rd mean curvature $H_3$ if and only if the function $\alpha$ defined by (\ref{alpha function}) is
constant on $M^{2n}:=\pi^{-1}(\widetilde{M})$;
 \item[c)] It is $3$-isoparametric if and only
 if the function $\alpha$ is constant on each (nearby) parallel hypersurface
 $M_t$ of $M^{2n}:=\pi^{-1}(\widetilde{M})$.
\end{itemize}
\end{cor}

\begin{proof}
It follows immediately from Proposition \ref{lift of isop to Sn},
Cartan's rigidity result (\ref{cartan result}), identities in
(\ref{sigma123ofS Stilde}), as well as Lemma \ref{geom. meaning for
kisop}.
\end{proof}

Suppose we are now given an isoparametric hypersurface $\widetilde{M}$ in $\mathbb{C}P^n$. Let $F:\mathbb{R}^{2n+2}\rightarrow \mathbb{R}$ be the isoparametric polynomial (satisfying Cartan-M\"{u}nzner equations
(\ref{eq1.3})-(\ref{eq1.4})) corresponding to the isoparametric
hypersurface $M=\pi^{-1}(\widetilde{M})\subset S^{2n+1}$, the
inverse image of $\widetilde{M}$, and $f=F|_{S^{2n+1}}$. Denote by the same symbol $J$ the matrix
representation of the corresponding complex structure $J$ in terms of the
Euclidean coordinates\footnote{Throughout this paper, by using the congruence
$\mathbb{R}^{N}\cong\mathcal{T}_x\mathbb{R}^{N}$, we identify $\partial x_k=\frac{\partial}{\partial x_k}=Dx_k$ with the k-th coordinate vector field $(0,\cdots,1,\cdots,0)$ for $1\leq k\leq N$, and $Dx^t=(\partial x_1,\cdots,\partial x_{2n+2})$ with the identity matrix $I$. The superscript $t$ means transposition, vectors are written in columns as points and also regarded as $(N\times1)$-matrices. The derivative $D$ gives a column vector when it acts on a function as gradient, and thus gives a matrix when it acts on a row vector of functions. A dot $``\cdot"$ between matrices means standard matrix product, and $\langle A, B\rangle:=tr(A^t\cdot B)$ denotes the inner product of $A,B$ in the $(m\times n)$ matrix space $M(m,n)$.}
$x=(x_1,\cdots,x_{2n+2})^t=\sum\limits_{k=1}^{2n+2}x_k\partial
x_k=Dx^t\cdot x$,
namely,
$$J(Dx^t):=J(\partial x_1,\cdots,\partial x_{2n+2})=(\partial
x_1,\cdots,\partial x_{2n+2})\cdot J=Dx^t\cdot J,$$
hence $J(V)=J\cdot V$, for any vector $V$ on $\mathbb{R}^{2n+2}$.

In these notations, we are ready to
give an explicit formula for the function $\alpha$ defined by
(\ref{alpha function}) on $M$ in terms of $F$ and $J$. Indeed,
$\alpha$ can be regarded as a function on $S^{2n+1}\backslash
\{M_{\pm}\}$ as follows.


\begin{prop}The function $\alpha$ on each parallel hypersurface $M_t:=f^{-1}(t)=F^{-1}(t)\cap
S^{2n+1}$ of $M$ can be described as a sum
\begin{equation}\label{formula for alpha}
\alpha= \frac{1}{g^3(1-F^2)^{\frac{3}{2}}}
\Big\{g^3F(3-2F^2)+\Omega_F\Big\},
\end{equation}
where $\Omega_F:=DF^t\cdot J\cdot D^2F\cdot J\cdot DF|_{S^{2n+1}}$
and $D^2F=D(DF^t)$ is the matrix of the Hessian $H_F$.
\end{prop}
\begin{proof}
At any point $p\in M_t\subset S^{2n+1}$, it is clear that the spherical gradient of $f$ is expressed by
\begin{equation*}
\nabla f(p)=DF-\langle DF, p\rangle p=DF-gfp,
\end{equation*}
and we can define the unit normal vector field of $M_t$ by
\begin{equation*}
\nu=\nabla f/|\nabla f|=\nabla f/\sqrt{b}.
\end{equation*} Thus $J(\nu)_p=\frac{1}{\sqrt{b}}J(DF-gfp)=\frac{1}{\sqrt{b}}J\cdot(DF-gfp)$,
 and by definition,
\begin{eqnarray}\label{alphaform}
&&\alpha|_p~=\langle\nu_p,\nabla_{J\nu_p}J\nu\rangle=\frac{1}{b\sqrt{b}}\langle DF-gfp,\nabla_{J(DF-gfp)}J(DF-gfx)\rangle\\
&&\qquad=\frac{1}{b\sqrt{b}}\langle DF-gfp, D_{J(DF-gfp)}J(DF-gFx)\rangle\nonumber\\
&&\qquad=\frac{1}{b\sqrt{b}}\langle DF-gfp, J\Big(D_{J(DF-gfp)}(DF-gFx)\Big)\rangle\nonumber\\
&&\qquad=\frac{1}{b\sqrt{b}}\langle DF-gfp, J\Big(D(DF-gFx)^t|_{x=p}\cdot J(DF-gfp)\Big)\rangle\nonumber\\
&&\qquad=\frac{1}{b\sqrt{b}}\langle DF-gfp, J\Big((D^2F-gDF\cdot x^t-gF Dx^t)|_{x=p}\cdot J(DF-gfp)\Big)\rangle\nonumber\\
&&\qquad=\frac{1}{b\sqrt{b}}\langle DF-gfp, J\cdot(D^2F-gDF\cdot p^t-gFI)\cdot J\cdot(DF-gfp)\rangle\nonumber\\
&&\qquad=\frac{1}{b\sqrt{b}}(DF^t-gfp^t)\cdot J\cdot(D^2F-gp\cdot DF^t-gFI)\cdot J\cdot(DF-gfp).\nonumber
\end{eqnarray}
where $x$ is the position vector field extending $p$. Note that
$M_t$ is $S^1$-invariant and thus $Jp\in\mathcal {T}_pM_t$, which
implies $DF^tJp=\langle\nabla f, Jp\rangle=0$ and thus $D^2F\cdot Jp
= JDF$. In addition, we have on hand several simple equalities:
\begin{eqnarray*}
&&J^2=-I,\quad p^tJp=0,\quad |\nabla f|^2=b=g^2(1-f^2),\nonumber\\
&&|DF|^2|_p=g^2, \quad p^t\cdot DF=gF.\nonumber
\end{eqnarray*} Applying these equalities, we conclude
\begin{eqnarray}
&&(DF^t-gFp^t)\cdot Jp\cdot DF^t= 0,\nonumber\\
&&(DF^t-gFp^t)\cdot J\cdot gFI\cdot J\cdot(DF-gFp)=-g^3f(1-f^2),\nonumber\\
&& (DF^t-gFp^t)\cdot J\cdot D^2F\cdot J\cdot (DF-gFp)=DF^t\cdot
J\cdot D^2F\cdot J\cdot DF + 2g^3F -g^3F^3.\nonumber
\end{eqnarray}
Substituting all these equalities in (\ref{alphaform}), we get
immediately the desired formula (\ref{formula for alpha}).
\end{proof}

Now we investigate the function $\alpha$ on the OT-FKM-type
isoparametric hypersurfaces in spheres which almost cover all
isoparametric hypersurfaces with four distinct principal curvatures
(cf. \cite{CCJ07}). For a symmetric Clifford system $\{A_0,\cdots,A_m\}$
on $\mathbb{R}^{2r}$, \emph{i.e.}, $A_i$'s are symmetric matrices
satisfying $A_iA_j+A_jA_i=2\delta_{ij}I_{2r}$, the OT-FKM-type
isoparametric polynomial $F$ on $\mathbb{R}^{2r}$ is then defined as
(cf.\cite{FKM}):
\begin{equation}\label{FKM isop. poly.}
F(z) = |z|^4 - 2\displaystyle\sum_{p = 0}^{m}{\langle
A_pz,z\rangle^2},
\end{equation}
where we take the coordinate system $z=(x^t,y^t)^t=(x_1,\cdots,x_r,y_1,\cdots,y_r)^t\in
\mathbb{R}^{2r}$.
By orthogonal transformations, we can write
\begin{eqnarray}\label{Clifford A_j}
&&A_0= \left(\begin{array}{c|c}
I & 0 \\
\hline 0 & -I
\end{array}\right),\quad
A_1= \left(\begin{array}{c|c}
0 & I \\
\hline I & 0
\end{array}\right),\\
&&A_j= \left(\begin{array}{c|c}
0 & -E_j \\
\hline E_j & 0
\end{array}\right),\quad j=2,\cdots,m,\nonumber
\end{eqnarray}
where $\{E_2,\cdots,E_m\}$ is a skew-symmetric Clifford system
on $\mathbb{R}^r$, \emph{i.e.}, $E_i$'s are skew-symmetric matrices
satisfying $E_iE_j+E_jE_i=-2\delta_{ij}I_r$. It can be verified that the level hypersurfaces
of this polynomial restricted to the unit sphere have $4$ distinct constant
principal curvatures with multiplicities $m_1=m$ and $m_2=r-m-1$, provided $r-m-1>0$.
Now fixing a complex structure $J$ on
$\mathbb{R}^{2r}$ under the coordinate system as $J=
\left(\begin{array}{c|c}
0 & -I \\
\hline I & 0
\end{array}\right)
$, we define the corresponding $S^1$-action on $\mathbb{R}^{2r}$ by
$e^{i\theta}\cdot z=(\cos\theta + \sqrt{-1}\sin\theta)z = \cos\theta
z+ \sin\theta Jz$. We prepare in advance the following equalities
which will be useful later.
\begin{eqnarray}\label{A_j identities}
&&A_0J = -JA_0 = -A_1,\quad A_1J = -JA_1 = A_0,\quad A_0A_1 = -A_1A_0 = -J,\\
&&A_jJ=JA_j, \quad for~~ j=2,\cdots,m.\nonumber
\end{eqnarray}


\begin{prop}\label{FKM Omega}
The OT-FKM-type isoparametric polynomial $F$ defined by (\ref{FKM
isop. poly.}) and (\ref{Clifford A_j}) is $S^1$-invariant under
the fixed complex structure $J$ and thus induces an isoparametric function $\tilde{f}$ on
$\mathbb{C}P^{r-1}$ through the Hopf fibration. Moreover, the
function $\Omega_F$ defined in (\ref{formula for alpha})
at the point $z\in S^{2r-1}$ can be expressed as a sum
\begin{eqnarray}\label{FKM Omega formula}
&&\frac{1}{64}\Omega_F=2F^2-F-2+8(1+F)\Big( \langle A_0z,z\rangle^2+\langle A_1z,z\rangle^2\Big)\\
&&\quad\qquad+16\sum_{q=2}^{m}\Big(\sum_{p=2}^m\langle A_pz,z\rangle\langle A_qz,JA_pz\rangle\Big)^2.\nonumber
\end{eqnarray}
\end{prop}


\begin{rem}
When $m=1$, $\Omega_F=64(-2F^2-F+2)$ and thus $\alpha$ is constant
on each level hypersurface of $f=F|_{S^{2r-1}}$. Consequently, it follows from
Proposition \ref{3isop-alpha} that the isoparametric
function $\tilde{f}$ on $\mathbb{C}P^{r-1}$ induced from $f$ is now
$3$-isoparametric.
\end{rem}

\begin{proof}
By a direct calculation using (\ref{A_j identities}), we have for $j\geq2$ that
\begin{equation*}
\langle A_jz,z\rangle=\langle A_jJz,Jz\rangle,\; \langle
A_jJz,z\rangle=0,
\end{equation*} which imply
\begin{equation*}
\langle A_je^{i\theta}z,e^{i\theta}z\rangle= \langle A_j(\cos\theta z +
\sin\theta Jz),\cos\theta z + \sin\theta Jz\rangle = \langle
A_jz,z\rangle.
\end{equation*}
Using
\begin{equation*}
\langle A_0e^{i\theta}z,e^{i\theta}z\rangle^2+\langle
A_1e^{i\theta}z,e^{i\theta}z\rangle^2=\langle A_0z,z\rangle ^2 +
\langle A_1z,z\rangle ^2,
\end{equation*} we verify the $S^1$-invariance of $F$,
\emph{i.e.}, $F(e^{i\theta}z)=F(z)$ for any $e^{i\theta}\in S^1$.

To compute $\Omega_F$, first we observe that
\begin{eqnarray}
&&\frac{1}{4}DF=|z|^2z-2\sum_{p = 0}^{m}{\langle A_pz,z\rangle
A_pz},\nonumber\\ && \frac{1}{4}D^2F=|z|^2I + 2zz^t - 2\sum_{p =
0}^{m}{\langle A_pz,z\rangle A_p} - 4\sum_{p =
0}^{m}{A_pzz^tA_p}.\nonumber
\end{eqnarray}
Then by definition,
\begin{eqnarray}
\frac{1}{64}\Omega_F&=& \Big(z^t-2\sum_{p = 0}^{m}{\langle
A_pz,z\rangle z^tA_p}\Big)\cdot J\cdot \Big(I + 2zz^t - 2\sum_{p =
0}^{m}{\langle A_pz,z\rangle A_p} - 4\sum_{p =
0}^{m}{A_pzz^tA_p}\Big)\nonumber\\&&\cdot J\cdot \Big(z-2\sum_{p =
0}^{m}{\langle A_pz,z\rangle A_pz}\Big),\nonumber
\end{eqnarray}
which will be calculated by $4$ parts as follows:\\

$(i)\quad\Big(z^t-2\displaystyle\sum_{p = 0}^{m}{\langle
A_pz,z\rangle z^tA_p}\Big)\cdot J\cdot I \cdot J\cdot
\Big(z-2\displaystyle\sum_{q = 0}^{m}{\langle A_qz,z\rangle
A_qz}\Big)$

\quad\quad $= \frac{1}{4}DF^t\cdot J\cdot I\cdot J\cdot
\frac{1}{4}DF $ $= -\frac{1}{16}|DF|^2 $ $= -1$; \vspace{2mm}

$(ii)\quad\Big(z^t-2\displaystyle\sum_{p = 0}^{m}{\langle
A_pz,z\rangle z^tA_p}\Big)\cdot J\cdot 2zz^t \cdot J\cdot
\Big(z-2\displaystyle\sum_{q = 0}^{m}{\langle A_qz,z\rangle
A_qz}\Big)$

\quad \quad$=8\displaystyle\sum_{p,q = 0}^{m}{\langle A_pz,z\rangle
\langle A_qz,z\rangle z^tA_pJ z\cdot
z^tJA_qz}$

\quad \quad$=-8\Big(\displaystyle\sum_{p= 0}^1{\langle A_pz,z\rangle
\langle JA_pz,z\rangle }\Big)^2=0$; \vspace{2mm}

$(iii)\quad \Big(z^t-2\displaystyle\sum_{p = 0}^{m}{\langle
A_pz,z\rangle z^tA_p}\Big)\cdot J\cdot \Big(-2\displaystyle\sum_{q =
0}^{m}{\langle A_qz,z\rangle A_q}\Big) \cdot J\cdot
\Big(z-2\displaystyle\sum_{j = 0}^{m}{\langle A_jz,z\rangle
A_jz}\Big)$

\quad \quad $=-2\displaystyle\sum_{q = 0}^{m}{\langle A_qz,z\rangle
}\Big(z^t-2\displaystyle\sum_{p = 0}^{m}{\langle A_pz,z\rangle
z^tA_p}\Big)\cdot JA_qJ\cdot \Big(z-2\displaystyle\sum_{j =
0}^{m}{\langle A_jz,z\rangle A_jz}\Big)$

\quad \quad $=-2\displaystyle\sum_{q = 0}^1{\langle A_qz,z\rangle
}\Big\{\Big(z^t-2\displaystyle\sum_{p = 0}^{m}{\langle A_pz,z\rangle
z^tA_p}\Big)\cdot A_q\cdot \Big(z-2\displaystyle\sum_{j =
0}^{m}{\langle A_jz,z\rangle A_jz}\Big)\Big\}$

\quad \quad\quad  $+ 2\displaystyle\sum_{q = 2}^{m}{\langle
A_qz,z\rangle }\Big\{\Big(z^t-2\displaystyle\sum_{p = 0}^{m}{\langle
A_pz,z\rangle z^tA_p}\Big)\cdot A_q\cdot
\Big(z-2\displaystyle\sum_{j = 0}^{m}{\langle A_jz,z\rangle
A_jz}\Big)\Big\}$

\quad \quad $=6\Big( \langle A_0z,z\rangle ^2+\langle A_1z,z\rangle
^2\Big) - 6\displaystyle\sum_{p = 2}^{m}{\langle A_pz,z\rangle ^2} $

\quad \quad $\quad -8\displaystyle\sum_{q = 0}^1{\langle
A_qz,z\rangle \Big(\langle A_qz,z\rangle ^3 +
\displaystyle\sum_{p\neq q}{\langle A_pz,z\rangle ^2\langle
A_qz,z\rangle }\Big)}$

\quad \quad\quad $+8\displaystyle\sum_{q = 2}^{m}{\langle
A_qz,z\rangle \Big(\langle A_qz,z\rangle ^3 +
\displaystyle\sum_{p\neq q}{\langle A_pz,z\rangle ^2\langle
A_qz,z\rangle }\Big)}$

\quad \quad $=-3(1-F) + 12\Big( \langle A_0z,z\rangle ^2+\langle
A_1z,z\rangle ^2\Big) +2(1-F)^2$

\quad \quad \quad$ - 8\Big(\langle A_0z,z\rangle ^2+\langle
A_1z,z\rangle ^2\Big)(1-F)$;

\vspace{2mm}

$(iv)\quad\Big(z^t-2\displaystyle\sum_{p = 0}^{m}{\langle
A_pz,z\rangle z^tA_p}\Big)\cdot J\cdot \Big(- 4\displaystyle\sum_{q
= 0}^{m}{A_qzz^tA_q}\Big) \cdot J\cdot \Big(z-2\displaystyle\sum_{j
= 0}^{m}{\langle A_jz,z\rangle A_jz}\Big)$

\quad\quad$=-4\displaystyle\sum_{q = 0}^1{z^tJA_qz\cdot z^tA_qJz} +
16\displaystyle\sum_{p,q}{\langle A_pz,z\rangle z^tA_pJA_qz\cdot
z^tA_qJz}$

\quad\quad$\quad -16\displaystyle\sum_{p,q,j}{\langle A_pz,z\rangle
\langle A_jz,z\rangle z^tA_pJA_qz\cdot z^tA_qJA_jz}$

\quad\quad$=4\Big( \langle A_0z,z\rangle ^2+\langle A_1z,z\rangle
^2\Big) + 16\displaystyle\sum_{p,q = 0}^1{\langle A_pz,z\rangle
\langle A_pz,JA_qz\rangle\langle A_qJz,z\rangle  }$

\quad\quad$\quad +
16\displaystyle\sum_{q=0}^{m}\Big(\sum_{p=0}^m\langle
A_pz,z\rangle\langle A_qz,JA_pz\rangle\Big)^2$

\quad\quad$=4\Big( \langle A_0z,z\rangle ^2+\langle A_1z,z\rangle
^2\Big) + 16\displaystyle\sum_{q=2}^{m}\Big(\sum_{p=2}^m\langle
A_pz,z\rangle\langle A_qz,JA_pz\rangle\Big)^2$.

Finally, taking sum of $(i)$, $(ii)$, $(iii)$, $(iv)$, we complete the proof
of the proposition.
\end{proof}


We are now ready to prove Theorem \ref{existence in CPn}.

\textbf{\emph{Proof of Theorem \ref{existence in CPn} (i).}} We
prefer to prove this assertion by making use of several known
results although there might be some direct approaches. The notations
remain the same as before.

First, a result of Park \cite{Pa} asserts that: If
$M^{2N}=\pi^{-1}(\widetilde{M}^{2N-1})$ is the inverse image of an
isoparametric hypersurface $\widetilde{M}^{2N-1}$ in
$\mathbb{C}P^N$, then the number $g$ of distinct principal
curvatures of the isoparametric hypersurface $M$ in $S^{2N+1}$ must be
$2$, $4$ or $6$; and if $g=6$, then the two multiplicities satisfy
$m_1=m_2=1$ and thus $N=3$ in this case. So when $N=2n$ is even, $g$
must be only $2$, or $4$. Hence it suffices to analyze these two cases
for our aim.

When $g=2$, Proposition $2.1$ in \cite{Xiao} stated that $\widetilde{M}$ has $2$ or
$3$ constant principal curvatures, thus is totally isoparametric
and homogeneous by equivalence sequence (\ref{equiv-CPC-homog}).

When $g=4$, firstly, according to Abresch \cite{Abr}, we can show
that either one of the multiplicities $\{m_1,m_2\}$ equals $1$, or $m_1=m_2=2$.
In fact, since $\emph{dim}_{R}\mathbb{C}P^{2n}=4n$, it follows that the corresponding
isoparametric hypersurface in the sphere is of $4n$ dimension, that is, $m_{1}+m_{2}=2n$.
Hence in the main theorem of \cite{Abr}, the case $4A$ is excluded; the case $4B_1$ occurs only when
$\min\{m_1,m_2\}=1$; and the case $4B_2$ occurs only when $m_1=m_2=2$, as we claimed.

In the first case, \emph{i.e.}, $\min\{m_1,m_2\}=1$, by virtue of
\cite{Tak2}, the isoparametric hypersurface $M^{4n}$ in the sphere
$S^{4n+1}$ must be homogeneous, and corresponds to the isotropy
representation of the rank two symmetric space
$W:=SO(2n+3)/S(O(2)\times O(2n+1))$, where $2n=m_{1}+m_{2}$. Let
$\mathfrak{o}(2n+3)=\mathfrak{k}\oplus\mathfrak{p}$ be the Cartan
decomposition, where $\mathfrak{k}$ is the Lie algebra of
$O(2)\times O(2n+1)$. Then the isoparametric hypersurface $M^{4n}$
is congruent to a principal orbit of the adjoint action of
$S(O(2)\times O(2n+1))$ on the vector space $\mathfrak{p}\cong
\mathbb{R}^{4n+2}$. In this representation, it is not difficult to
show that there is a unique complex structure (up to a sign, the
standard complex structure) on $\mathfrak{p}$ such that $M^{4n}$ is
$S^1$-invariant with respect to this complex structure, as
\cite{Xiao} claimed. This property helps us deduce that the number
$l$ of non-horizontal principal eigenspaces of $M^{4n}$ equals $2$
identically. Hence by the equivalence sequence
(\ref{equiv-CPC-homog}), $\widetilde{M}^{4n-1}$ is totally
isoparametric and homogeneous.

At last, we need to prove that the second case, \emph{i.e.},
$m_1=m_2=2$, is impossible. Without loss of generality, we can
assume that $M^{4n}$ is compact. Recall that a topological theorem
of M\"{u}nzner \cite{Mu80} determines the cohomology rings of a
compact isoparametric hypersurface in a sphere. By applying it, we
have $H^{q}(M^{4n},\mathbb{Z}_2)=0$ for any odd number $q$,
$H^{0}(M^{4n},\mathbb{Z}_2)=H^{4n}(M^{4n},\mathbb{Z}_2)=\mathbb{Z}_2$,
and $H^{2k}(M^{4n},\mathbb{Z}_2)=\mathbb{Z}_2\oplus\mathbb{Z}_2$,
for $k=1,\cdots,2n-1$. It follows from Poincar\'{e} duality that the
Euler characteristic $\chi(M)$ of $M^{4n}$ is equal to $2g=8>0$. On
the other hand, since $M^{4n}$ is the inverse image of a hypersurface
$\widetilde{M}^{4n-1}$ in $\mathbb{C}P^{2n}$, $Jx$ ($x$ is the
position vector of $M$) is a globally defined tangent vector field
without singularities on $M$ and thus by the Hopf index theorem, the
Euler characteristic $\chi(M)=0$, a contradiction which completes
the proof of Theorem \ref{existence in CPn}(i). \hfill $\Box$


\textbf{\emph{Proof of Theorem \ref{existence in CPn} (ii).}} By
Proposition \ref{lift of isop to Sn} and Corollary
\ref{3isop-alpha}, it suffices to construct a required
$S^1$-invariant isoparametric hypersurface $M$ (resp. isoparametric
polynomial $F$) in $S^{4n+3}$ (resp. on $\mathbb{R}^{4n+4}$) such
that the function $\alpha$ is non-constant on $M$. More
specifically, because of Proposition \ref{FKM Omega} we will look
for symmetric Clifford systems $\{A_0,A_1,\cdots,A_m\}$ on
$\mathbb{R}^{4n+4}$ in the form (\ref{Clifford A_j}), such that the
function $\Omega_F$ for the corresponding OT-FKM-type isoparametric
polynomial $F$ defined by (\ref{FKM isop. poly.}), could be not only
computed explicitly by formula (\ref{FKM Omega formula}), but also
non-constant on some level hypersurface $M$ of $F|_{S^{4n+3}}$.
Towards the aim, the first non-trivial case is when $m=2$, where we
successfully find an example for each $n\geq1$.

Note that when $m=2$, the formula (\ref{FKM Omega formula}) can be
deduced to
\begin{equation}\label{Omega m 2}
\Omega_F=64\Big(-2F^2-F+2-8(1+F)\langle A_2z,z\rangle^2\Big).
\end{equation}
Let $E_2= \left(\begin{array}{c|c}
0 & -I_{n+1} \\
\hline I_{n+1} & 0
\end{array}\right)$
be the sub-matrix of $A_2$ in (\ref{Clifford A_j}) with $2r=4n+4$.
Then it is easily verified that $\{A_0,A_1,A_2\}$ is now a symmetric
Clifford system .

Now for
$z=(x^t,y^t)^t=(x_1,\cdots,x_{2n+2},y_1,\cdots,y_{2n+2})^t\in
\mathbb{R}^{4n+4}$ ($n\geq1$), the OT-FKM-type isoparametric
polynomial $F$ can be written as
\begin{equation}\label{FKM example of N3isop with m 2}
F(z)=|z|^4-2\Big\{(|x|^2-|y|^2)^2+4\langle x,y\rangle^2+4\langle
E_2x,y\rangle^2\Big\}.
\end{equation}
 Let $M=F^{-1}(0)\bigcap
S^{4n+3}$, and $z=(x^t,y^t)^t$,
$\check{z}=(\check{x}^t,\check{y}^t)^t$ be two points in $M$ with
$x_1=\check{x}_1=\frac{1}{\sqrt{2}}$,
$y_1=y_2=\check{y}_{n+2}=\check{y}_{n+3}=\frac{1}{2}$ and the other
coordinates vanishing. Then it is easily calculated that
$\Omega_F(z)=128$ and $\Omega_F(\check{z})=-128$. Equivalently,
$\alpha$ is non-constant on $M$. Therefore, the isoparametric
hypersurface $\widetilde{M}=\pi(M)$ in $\mathbb{C}P^{2n+1}$ is not
$3$-isoparametric, as desired. \hfill $\Box$

It is worthy remarking that the isoparametric polynomial $F$ in (\ref{FKM example of N3isop with m 2}) also
induces a homogeneous hypersurface in $\mathbb{C}P^{2n+1}$ under
some other complex structure, by comparing Takagi's (\cite{Tak})
classification. Next, we calculate the function $\alpha$ (or
equivalently $\Omega_F$ defined in (\ref{formula for alpha}))
explicitly for the inhomogeneous example of Ozeki-Takeuchi
\cite{OT75} with $g=4$ and multiplicities $(m_1,m_2)=(3,4r)$ under
two different complex structures. Both functions turn out non-constant
on a level hypersurface in the sphere. As a result, these two induced
isoparametric hypersurfaces in $\mathbb{C}P^{4r+3}$ are not
$3$-isoparametric. From another point of view, by comparing Takagi's classification, we
know that these two induced isoparametric hypersurfaces in
$\mathbb{C}P^{4r+3}$ are not homogeneous, and hence by Theorem
\ref{H3classification in Cpn} they are not $3$-isoparametric.


\textbf{\emph{More Examples.}} First we decompose the quaternionic
space $\mathbb{H}^{2r+2}\cong \mathbb{R}^{8r+8}$ ($r\geq1$) as
$\mathbb{H}^{2r+2}=(\mathbb{H}\times
\mathbb{H}^{r})\times(\mathbb{H}\times \mathbb{H}^{r})$,
\emph{i.e.,} for
\begin{equation*}
z=(x_1,\cdots,x_{4r+4},y_1,\cdots,y_{4r+4})^t\in\mathbb{R}^{8r+8}\cong\mathbb{H}^{2r+2},
\end{equation*} we write $z=(u^{t},v^{t})^{t}$, where
$u=(u_0,\hat{u}) \in \mathbb{H}\times \mathbb{H}^{r}$,
$v=(v_0,\hat{v}) \in \mathbb{H}\times \mathbb{H}^{r}$,
$\hat{u}=(u_1,...,u_r)$, $\hat{v}=(v_1,...,v_r)$, and
\begin{eqnarray*}
&&u_i=x_{4i+1}+x_{4i+2}\mathbf{i}+x_{4i+3}\mathbf{j}+x_{4i+4}\mathbf{k}\in\mathbb{H}\cong\mathbb{R}^4,\nonumber\\
&&v_i=y_{4i+1}+y_{4i+2}\mathbf{i}+y_{4i+3}\mathbf{j}+y_{4i+4}\mathbf{k}\in\mathbb{H}\cong\mathbb{R}^4.\nonumber
\end{eqnarray*}
Then the isoparametric polynomial $F$ of the inhomogeneous example
of Ozeki-Takeuchi \cite{OT75} with $g=4$ and multiplicities
$(m_1,m_2)=(3,4r)$ is defined by
\begin{equation}\label{OT 34}
F(z)=|z|^4-2\Big\{4\Big(|u\cdot \bar{v}^t|^2-\langle u,v \rangle
^2\Big)+ \Big(|\hat{u}|^2 -|\hat{v}|^2 + 2\langle u_0, v_0 \rangle
\Big)^2\Big\},
\end{equation}
where the canonical involution
$\bar{v}_i=y_{4i+1}-y_{4i+2}\mathbf{i}-y_{4i+3}\mathbf{j}-y_{4i+4}\mathbf{k}$
and the quaternionic multiplication in $\mathbb{H}$ are used.

Let
\begin{equation*}
A_0=
\begin{pmatrix}
 0& & I_4&  \\
 & I_{4r}&  & 0\\
 I_{4}& & 0 &\\
  &0 &  & -I_{4r}
\end{pmatrix}
, A_p=  \left(\begin{smallmatrix}
 & & & D_p&  &  \\
 & & &  & \ddots & \\
 & & &  & & D_p\\
-D_p& & & & &  \\
  & \ddots & & & &\\
  & &-D_p& & &
\end{smallmatrix}\right)
,~~for ~~p=1,2,3,
\end{equation*} \vspace{2mm}
where~~ $D_1=  \left(\begin{smallmatrix}
0& -1 & & \\
1& 0 & &\\
 & & 0 & -1 \\
 & & 1& 0
\end{smallmatrix}\right)
$,~ $D_2=  \left(\begin{smallmatrix}
 & & -1& 0  \\
 & & 0 & 1  \\
 1& 0&  &  \\
 0&-1 & &
\end{smallmatrix}\right)
$,~and~~ $D_3=  \left(\begin{smallmatrix}
 & & 0& -1  \\
 & & -1 & 0  \\
 0& 1 &  &  \\
 1& 0 & &
\end{smallmatrix}\right)
$. Then a straightforward calculation shows that
$\{A_0,A_1,A_2,A_3\}$ is a symmetric Clifford system (though not in
the form (\ref{Clifford A_j})) and the polynomial $F$ defined by
(\ref{OT 34}) can also be expressed as the OT-FKM-type isoparametric
polynomial:
\begin{equation}\label{OT34FKM}
F=|z|^4-2\displaystyle\sum_{p = 0}^{3}{\langle A_pz,z \rangle ^2}.
\end{equation}

In the following, we will calculate the function $\alpha$ under two different complex structures:

(i) Let $J$ be the complex structure on
$\mathbb{R}^{8r+8}\cong\mathbb{H}^{2r+2}$ as the orthogonal
transformation induced by the right multiplication of $\mathbf{i}$
whose matrix representation is
\begin{equation}\label{cplx str r i}
 J= \left(
\begin{smallmatrix}
D_0 &  &  &\\
    & D_0 &  &\\
    &  & \ddots & \\
    &  &  & D_0
\end{smallmatrix} \right),\quad
\emph{where} \quad D_0= \left(\begin{smallmatrix}
0& -1 & & \\
1& 0 & &\\
 & & 0 & 1 \\
 & & -1& 0
\end{smallmatrix}\right).
\end{equation}
Evidently, $F$ is $S^1$-invariant under this complex
structure and $f=F|_{S^{8r+7}}$ thus induces an isoparametric
function $\tilde{f}$ on $\mathbb{C}P^{4r+3}$ through the
corresponding Hopf fibration.

By direct calculations , we have the following relations
\begin{eqnarray}
&&D_0D_p=D_pD_0,\quad JA_p = A_pJ, \quad for\quad p=0,1,2,3,\nonumber\\
&&D_1D_2=-D_2D_1=D_3,\quad D_2D_3=-D_3D_2=D_1,\quad
D_3D_1=-D_1D_3=D_2.\nonumber
\end{eqnarray}
Therefore,
$$(JA_p)^t = -JA_p,\quad (A_pJA_q)^t=A_pJA_q, \quad
for\quad p, q =0,1,2,3,~~ p\neq q,$$
which help us deduce the formula
for $\Omega_F$ (defined in (\ref{formula for alpha})) under $J$ in
(\ref{cplx str r i}) as
\begin{equation*}
\Omega_F=64\Big\{2F^2-F-2 +16\displaystyle\sum_{q =
0}^{3}{\Big(\displaystyle\sum_{p = 0}^{3}{\langle A_pz,z \rangle
\langle JA_qz,A_pz \rangle }\Big)^2}\Big\}.
\end{equation*}

(ii) Let $J'$ be another complex structure on
$\mathbb{R}^{8r+8}\cong\mathbb{H}^{2r+2}$ as the orthogonal
transformation induced by the left multiplication of $\mathbf{i}$
whose matrix representation is
\begin{equation}\label{cplx str l i}
J'= \left(
\begin{smallmatrix}
D_1 &  &  &\\
    & D_1 &  &\\
    &  & \ddots & \\
    &  &  & D_1
\end{smallmatrix} \right),\quad \emph{where}~ D_1~\emph{was given before}.
\end{equation}
Similarly, $F$ is also $S^1$-invariant under $J'$ and thus induces an
isoparametric function $\tilde{f}'$ on $\mathbb{C}P^{4r+3}$ through
the corresponding Hopf fibration.

Again, by direct calculations , we have the following relations
\begin{eqnarray}
&&J'A_0=A_0J',\quad J'A_1=A_1J'\nonumber\\
&&J'A_2=-A_2J'=A_3, \quad J'A_3=-A_3J'=-A_2,\nonumber
\end{eqnarray}
and hence $A_pJ'A_q$ is skew-symmetric for almost all $p\neq q$
except for $A_0J'A_1$, $A_1J'A_0$, $A_2J'A_3=-I$, $A_3J'A_2=I$ which
are symmetric. Then we can deduce the formula for $\Omega_F$
(defined in (\ref{formula for alpha})) under $J'$ in (\ref{cplx str
l i}) as
\begin{eqnarray*}
&&\frac{1}{64}\Omega_F=2F^2-F-2 + 8(1+F)\Big(\langle A_2z,z \rangle
^2+\langle A_3z,z \rangle ^2\Big)\\
&&\ \quad\qquad+16 \Big(\langle A_0z,z \rangle
^2+\langle A_1z,z \rangle ^2\Big)\langle A_0J'A_1z,z \rangle
^2.
\end{eqnarray*}

In conclusion, let $M=F^{-1}(0)\bigcap S^{8r+7}$, and
$z=(x^t,y^t)^t$, $\check{z}=(\check{x}^t,\check{y}^t)^t$ be two
points in $M$ with $x_1=\frac{1}{2}\sqrt{2+\sqrt{2}}$,
$y_1=\check{y}_5=\frac{1}{2}\sqrt{2-\sqrt{2}}$,
$\check{x}_1=\check{y}_1=\frac{1}{2\sqrt{2}}\sqrt{2+\sqrt{2}}$, and
the other coordinates vanishing. Then it is easily calculated that,
under both complex structures $J$, $J'$ defined in (i), (ii) above,
$\Omega_F(z)=128$ and $\Omega_F(\check{z})=-128$. Therefore,
$\alpha$ is non-constant on $M$ under both $J$ and $J'$. This means
that the isoparametric hypersurfaces $\widetilde{M}=\pi(M)$,
$\widetilde{M}'=\pi'(M)$ in $\mathbb{C}P^{4r+3}$ are not
$3$-isoparametric, where $\pi$ and $\pi'$ are the corresponding Hopf
fibrations ${S}^{8r+7}\longrightarrow {\mathbb{C}P^{4r+3}}$ induced
by $J$ and $J'$, respectively.\hfill $\Box$


To conclude this section, we come to prove Theorem
\ref{H3classification in Cpn}.

\textbf{\emph{Proof of Theorem \ref{H3classification in Cpn}.}} Suppose that
$\widetilde{M}^{2n-1}$ is an isoparametric hypersurface in
$\mathbb{C}P^n$ of constant $3$rd mean curvature $H_3$ with unit
normal vector field $\tilde{\nu}$. Then
$M^{2n}=\pi^{-1}(\widetilde{M}^{2n-1})$ is an isoparametric
hypersurface in $S^{2n+1}$. As we pointed out before, it has $g=2,4, or~6$ distinct constant
principal curvatures $\lambda_1>\cdots>\lambda_g$. Let
$T_{\lambda_i}$ be the principal distribution on $M$ corresponding
to $\lambda_i$ and thus $\mathcal
{T}M=T_{\lambda_1}\oplus\cdots\oplus T_{\lambda_g}$. By Corollary
\ref{3isop-alpha}, $\alpha:=\langle S_{\nu}J\nu,J\nu\rangle$ is now
constant on $M$. We will use the same notations as those at
the beginning of this section.

Let $x$ be the position vector field of $M$. Then $Jx$ is the
vertical vector field tangent to the $S^1$-fibres of the Hopf
fibration. Represent $Jx$ as
\begin{equation}\label{Jx decomposition}
Jx=\phi_1\epsilon_1+\cdots+\phi_g\epsilon_g,
\end{equation}
where $\epsilon_i\in T_{\lambda_i}$ is a unit vector and
$\phi_i\geq0$ is the length of the component of $Jx$ in
$T_{\lambda_i}$ for $i=1,\cdots,g$. It follows from (\ref{SJx}) and
(\ref{Jx decomposition}) that
\begin{equation}\label{SJv decom}
S_{\nu}Jx=\lambda_1\phi_1\epsilon_1+\cdots+\lambda_g\phi_g\epsilon_g=-J\nu,
\end{equation}
which together with the fact that $Jx,J\nu$ are orthogonal unit
vectors implies
\begin{equation}\label{l computation}
\phi_1^2+\cdots+\phi_g^2=1,\quad
\lambda_1\phi_1^2+\cdots+\lambda_g\phi_g^2=0,\quad
\lambda_1^2\phi_1^2+\cdots+\lambda_g^2\phi_g^2=1.
\end{equation}
Similarly, by (\ref{SJv decom}) we have
\begin{equation}\label{alpha decom}
\alpha=\langle
S_{\nu}J\nu,J\nu\rangle=\lambda_1^3\phi_1^2+\cdots+\lambda_g^3\phi_g^2.
\end{equation}
Note that the number $l$ of non-horizontal eigenspaces of $S_{\nu}$
equals the number of non-zero $\phi_i$'s. By the equivalence sequence (\ref{equiv-CPC-homog}), it suffices to prove
$l\equiv const$ case by case with respect to $g=2,4,or~6$.

(i) When $g=2$, it follows immediately from (\ref{l computation})
that
$$\phi_1=\sqrt{\frac{-\lambda_2}{\lambda_1-\lambda_2}}, \quad \phi_2=\sqrt{\frac{\lambda_1}{\lambda_1-\lambda_2}},\quad \lambda_1=-\frac{1}{\lambda_2}>0,$$
which imply $\alpha=\lambda_1+\lambda_2\equiv const$, $l\equiv2$ and
thus $\widetilde{M}$ is homogeneous by (\ref{equiv-CPC-homog})(see
also \cite{Xiao}).

(ii) When $g=4$, it follows immediately from (\ref{l computation})
and (\ref{alpha decom}) that
$$\begin{pmatrix}
 \phi_1^2\\
 \phi_2^2\\
 \phi_3^2\\
 \phi_4^2
\end{pmatrix}=\begin{pmatrix}1&1&1&1\\ \lambda_1&\lambda_2&\lambda_3&\lambda_4\\
\lambda_1^2&\lambda_2^2&\lambda_3^2&\lambda_4^2\\
\lambda_1^3&\lambda_2^3&\lambda_3^3&\lambda_4^3
\end{pmatrix}^{-1}\begin{pmatrix}
 1\\0\\1\\\alpha
\end{pmatrix},$$
which implies that $l\equiv const$ provided $\alpha\equiv const$ and
thus $\widetilde{M}$ is homogeneous by (\ref{equiv-CPC-homog}).

(iii) When $g=6$, as mentioned before,  \cite{Pa} proved that
$m_1=m_2=m$ must be 1. We need only to prove that $\alpha$ is always
non-constant in this case. By virtue of \cite{Xiao}, for any
$c_1,c_2,c_3$ satisfying $c_1^2+c_2^2+c_3^2=1$, there is a point
$x\in M$ such that
$$\phi_i(x)^2=\frac{c_i^2}{1+\lambda_i^2},\quad \phi_{i+3}(x)^2=\frac{c_i^2}{1+\lambda_{i+3}^2},\quad for~~i=1,2,3.$$
Then
$$\alpha(x)=(\lambda_1+\lambda_4)c_1^2+(\lambda_2+\lambda_5)c_2^2+(\lambda_3+\lambda_6)c_3^2.$$
Now let $c_1=1,~c_2=c_3=0$ and $x'\in M$ be the corresponding point.
Then $\alpha(x')=\lambda_1+\lambda_4$. Similarly, let
$c_1=c_3=0,~c_2=1$ and $x''\in M$ be the corresponding point. Then
$\alpha(x'')=\lambda_2+\lambda_5$. Therefore, $\alpha$ is always
non-constant on $M$ as we required.

The proof is now complete. \hfill $\Box$


\textbf{\emph{Proof of Corollary \ref{Chern exam}.}} Assume that
$\widetilde{M}$ is an inhomogeneous hypersurface in $\mathbb{C}P^n$
with constant $1$st, $2$nd, and $3$rd mean curvatures $H_1,H_2,H_3$.
Then by equalities (\ref{sigma123ofS Stilde}), the inverse image
$M=\pi^{-1}(\widetilde{M})$ in $S^{2n+1}$ under the Hopf fibration
has constant $1$st mean curvature $H_1$ and constant $2$nd mean
curvature $H_2-1$. It suffices to show that $M$ is not an
isoparametric hypersurface. We will prove this by contradiction.

Suppose $M$ is isoparametric. It follows from Proposition \ref{lift of isop to
Sn} that $\widetilde{M}$ is also an isoparametric hypersurface in
$\mathbb{C}P^n$. Since $\widetilde{M}$ has constant $3$rd mean
curvature by assumption, it is homogeneous by Theorem \ref{H3classification in Cpn},
which contradicts the assumption that $\widetilde{M}$ is an inhomogeneous hypersurface.
 \hfill $\Box$

\section{Isoparametric hypersurfaces in rank one symmetric spaces}
In this section, by using the Riccati equation we first derive some
``weakly" inductive formulae for $Q_k:=\rho_k(S(t))$ on parallel
hypersurfaces $M_t$ in a general Riemannian manifold in the spirit
of Theorem \ref{implicit CM eqs}. Next, by using further symmetries of
the Jacobi operator on a complex space form, more generally, on a locally
rank one symmetric space, we will prove Theorem
\ref{3isop rk one sp} and Theorem \ref{compatible rk one sp}.

For our purpose, let $\{M_t: t\in(-\varepsilon,\varepsilon)\}$ be a family of
parallel hypersurfaces in a Riemannian manifold $N^{n+1}$,
$\nu_t$ the unit normal vector field on $M_t$, $S(t)=S_{\nu_t}$ the
shape operator of $M_t$, and $R(t)=K_{\nu_t}$ the normal
Jacobi operator (see definition in Remark \ref{Jacobi remark}) on
$M_t$. It is convenient to denote the covariant derivatives of the operators $S(t)$,
$R(t)$ along normal geodesics of $M_t$ by
$S'(t):=\nabla_{\nu_t}S(t)$, $R'(t):=\nabla_{\nu_t}R(t)$,
respectively. In this way, the well known Riccati equation can be given by
(cf. \cite{Gr04}):
\begin{equation}\label{Riccati eq}
S'(t)=S(t)^2+R(t).
\end{equation}
By taking trace with respect to a parallel orthonormal frame
$\{E_1,\cdots,E_n\}$ along a normal geodesic of $M_t$, we get the
Riccati equation for the mean curvature $H(t)$ of $M_t$:
\begin{equation}\label{Riccati eq for mean curv}
H'(t)=\|S(t)\|^2+Ric(t),
\end{equation}
where $Ric(t)=Ric(\nu_t,\nu_t)$ denotes the Ricci curvature of $N$ in the
normal direction of $M_t$.


\begin{prop}\label{1 isop 2 isop}
A $1$-isoparametric hypersurface in an Einstein manifold $N^{n+1}$
must be $2$-isoparametric.
\end{prop}

\begin{proof}
Observe that $H(t)$ is now a function depending only on $t$, and
hence $H'(t)$ is constant on $M_t$. Since $Ric(t)\equiv \rho$ the
Einstein constant, the conclusion follows immediately from the
equality (46), Lemma \ref{geom. meaning for kisop} and Newton's
identities (\ref{Newton identity}).
\end{proof}
For $i\geq0,~j\geq0$, $t\in(-\varepsilon,\varepsilon)$, we introduce a
function $\Gamma_{ij}(t)$ on $M_t$ by
\begin{equation}\label{Gammaij}
\Gamma_{ij}(t):=tr\Big(S(t)^iR(t)^j\Big).
\end{equation}
Clearly, $\Gamma_{i0}(t)=tr(S(t)^i)=\rho_i(S(t))=:Q_i(t)$. As
discussed before, applying Lemma \ref{geom. meaning for kisop} and
Newton's identities (\ref{Newton identity}), once we find some
inductive formulae for the $Q_i$'s on $M_t$ as those in Theorem
\ref{implicit CM eqs}, we would establish similar rigidity results
with Cartan's rigidity result (\ref{cartan result}), Theorem
\ref{existence in CPn} (i) and Theorem \ref{H3classification in
Cpn}. Towards this aim, we simply take a derivative of $Q_i$ with
respect to $t$ to obtain:


\begin{lem} With notations as above,
\begin{equation}\label{inductive Qi general}
Q_{i+1}(t)=\frac{1}{i}Q_i'(t)-\Gamma_{i-1, 1}(t).
\end{equation}
\end{lem}

\begin{proof}
It follows directly from the definitions and the Riccati equation
(\ref{Riccati eq}) that
\begin{eqnarray*}
&&Q_i'(t)=tr\Big(\nabla_{\nu_t}S(t)^i\Big)=\sum_{j=0}^{i-1}tr\Big(S(t)^j\Big(\nabla_{\nu_t}S(t)\Big)S(t)^{i-1-j}\Big)\nonumber\\
&&\qquad\ =i~tr\Big(S(t)^{i-1}S'(t)\Big)=i\Big(tr(S(t)^{i+1})+ tr(S(t)^{i-1}R(t))\Big)\\
&&\qquad\ =i~\Big(Q_{i+1}(t)+\Gamma_{i-1, 1}(t)\Big)\nonumber
\end{eqnarray*}
as required.
\end{proof}

However, the inductive formula (\ref{inductive Qi general}) does not
work effectively unless the functions $\Gamma_{i-1, 1}(t)$ are
constant on $M_t$. So we need to investigate some inductive
properties of the functions $\Gamma_{i-1, 1}(t)$ as the following.


\begin{lem} With notations as above,
\begin{equation}\label{inductive Gamma general}
\Gamma_{i+1, 1}(t)=\frac{1}{i}\Big\{\Gamma_{i1}'(t)-\sum_{j=0}^{i-1}tr\Big(S(t)^jR(t)S(t)^{i-1-j}R(t)\Big)-tr\Big(S(t)^iR'(t)\Big)\Big\}.
\end{equation}
\end{lem}

\begin{proof}
Similarly, it follows directly from the definitions and the Riccati equation
(\ref{Riccati eq}) that
\begin{eqnarray}
\Gamma_{i1}'(t)&=&\sum_{j=0}^{i-1}tr\Big(S(t)^jS'(t)S(t)^{i-1-j}R(t)\Big)+tr\Big(S(t)^iR'(t)\Big)\nonumber\\
&=&i~\Gamma_{i+1, 1}(t)+\sum_{j=0}^{i-1}tr\Big(S(t)^jR(t)S(t)^{i-1-j}R(t)\Big)+tr\Big(S(t)^iR'(t)\Big).\nonumber
\end{eqnarray}
\end{proof}

In order to make these inductive formulae work effectively, we put
some restrictions on the ambient manifold so as to control the last
two terms on the right side of (\ref{inductive Gamma general}). To
be more precise, we have the following result stated in Remark
\ref{Jacobi remark}.


\begin{cor}\label{Jacobi cor}
Let $N^{n+1}$ be a locally symmetric space with the property that,
the first and second elementary symmetric polynomials on eigenvalues
of its Jacobi operator $K_{\xi}$ are constant and independent of the
choice of the unit tangent vector $\xi\in\mathcal {T}N$. Then any
$1$-isoparametric hypersurface in $N$ must be $2$-isoparametric,
additionally, any $3$-isoparametric hypersurface in $N$ must be
$4$-isoparametric.
\end{cor}

\begin{proof} The first assertion is a consequence of Proposition \ref{1 isop 2
isop}, since $N$ is now an Einstein manifold. To prove the second
assertion, it suffices to show that
\begin{equation*}
Q_4(t)=\Gamma_{40}(t)=tr(S(t)^4)
\end{equation*} is constant on each nearby
parallel hypersurface $M_t$ of a $3$-isoparametric hypersurface $M$
in $N$ for $t\in (-\varepsilon,\varepsilon)$. Since $M$ is
$3$-isoparametric, $Q_1(t)$, $Q_2(t)$ and $Q_3(t)$ are all constant on
$M_t$ and thus smooth functions depending only on $t$. Then by
(\ref{inductive Qi general}),
\begin{equation*}
\Gamma_{11}(t)=\frac{1}{2}Q_2'(t)-Q_3(t)
\end{equation*} is a smooth function
depending only on $t$. Consequently, by (\ref{inductive Gamma
general}),
\begin{equation*}
\Gamma_{21}(t)=\Gamma_{11}'(t)-tr(R(t)^2)
\end{equation*} is a smooth
function depending only on $t$, since by assumptions we have
\begin{equation*}
tr(R(t)^2)\equiv const~~ and~~ R'(t)\equiv0.
\end{equation*} Again by
(\ref{inductive Qi general}),
\begin{equation*}
Q_4(t)=\frac{1}{3}Q_3'(t)-\Gamma_{21}(t)
\end{equation*} is a smooth function
depending only on $t$, which completes the proof.
\end{proof}

As mentioned in Remark \ref{Jacobi remark}, the restrictions we put
in this corollary is not so strong that, there exist many locally
symmetric spaces with rank greater than one satisfying these
conditions. Now we will be concerned with the locally rank one
symmetric spaces $N^{n+1}$. Obviously, the Jacobi operator
$K_{\xi}$ has constant eigenvalues independent of the choice of the
unit tangent vector $\xi\in\mathcal {T}N$. This property will be useful
in the establishment of further rigidity results. To warm up before the
proof of the theorems, we deal with the case when $K_{\xi}$ has only
one constant eigenvalue $c$ besides the trivial $0$-eigenvalue,
\emph{i.e.}, when $N^{n+1}$ is a real space form with constant
sectional curvature $c$. Now, we derive Cartan's rigidity result
(\ref{cartan result}) as a simple application of the inductive
formulae (\ref{inductive Qi general})-(\ref{inductive Gamma general}).


\textbf{\emph{Proof of (\ref{cartan result}).}} With notations as
before, $R(t)\equiv cI$, $\Gamma_{i1}(t)=cQ_i(t)$ and $R'(t)\equiv0$
under any orthonormal frame of $M_t$. Then either of (\ref{inductive
Qi general}) and (\ref{inductive Gamma general}) can be deduced to
$$Q_{i+1}(t)=\frac{1}{i}Q_i'(t)-cQ_{i-1}(t),$$
which immediately implies (\ref{cartan result}) by induction (note
that the preceding formula differs from that in Theorem \ref{implicit CM
eqs} as the parameter $t$ has different meanings). \hfill $\Box$


\textbf{\emph{Proof of Theorem \ref{3isop rk one sp}.}} Obviously we
need only to consider the case when $N^{n+1}$ is a locally rank one
symmetric space with non-constant sectional curvatures. In this case
the Jacobi operator $K_{\xi}$ has two distinct non-zero constant
eigenvalues $\kappa_1,\kappa_2$ independent of $\xi$. Let $M^n$ be a
$3$-isoparametric hypersurface in $N^{n+1}$ and $M_t$
$(t\in(-\varepsilon,\varepsilon))$ nearby parallel hypersurfaces
with unit normal vector fields $\nu_t$. Then there exists a local
orthonormal frame $\{e_1(t),\cdots,e_n(t)\}$ of $M_t$ parallel
along normal geodesics of $M$, \emph{i.e.},
$\nabla_{\nu_t}e_i(t)=0$, such that under this frame,
\begin{equation}\label{Jacobi rk one}
R(t)=K_{\nu_t}|_{\mathcal
{T}M_t}=diag(\kappa_1I_{n-m},~\kappa_2I_{m}),
\end{equation}
 where $m=1,3,or~7$
corresponding to the case when $N$ is locally a complex space form,
a quaternionic space form, or an octonionic space form,
respectively. Therefore, as in the proof of Corollary \ref{Jacobi
cor}, $Q_1(t),Q_2(t),Q_3(t),\Gamma_{11}(t),\Gamma_{21}(t),Q_4(t)$
are all smooth functions depending only on $t$. Decompose $S(t)^i$
into the same blocks as $R(t)$ above for $i\geq0$, namely,
\begin{equation}\label{blocks of Si}
S(t)^i=\left(\begin{array}{c|c}
A_i & B_i \\
\hline B_i^t & C_i
\end{array}\right), \quad A_0=I_{n-m},~~C_0=I_m,~~B_0=0,
\end{equation}
where $A_i,C_i$ are symmetric matrices of order $n-m$ and $m$,
respectively. These arguments yield
\begin{equation}\label{blocks}
Q_i(t)=tr(A_i)+tr(C_i),\quad
\Gamma_{ij}(t)=\kappa_1^j~tr(A_i)+\kappa_2^j~tr(C_i).
\end{equation}
Notice that $\kappa_1$, $\kappa_2$ are two distinct constants. It follows that if
for some $i,j\geq1$, $Q_i(t)$, $\Gamma_{ij}(t)$ are smooth functions
depending only on $t$, then by (\ref{blocks}), $tr(A_i)$, $tr(C_i)$
and hence $\Gamma_{ik}(t)$ are also smooth functions depending only
on $t$ for any $k\geq1$. Now since
$Q_1(t),Q_2(t),\Gamma_{11}(t),\Gamma_{21}(t)$ are such functions,
$\Gamma_{1k}(t)$, $\Gamma_{2k}(t)$ are smooth functions depending
only on $t$ for any $k\geq1$. Then by (\ref{inductive Qi general}), (\ref{inductive Gamma
general}) again,
\begin{equation*}
\Gamma_{31}(t)=\frac{1}{2}\Gamma_{21}'(t)-\Gamma_{12}(t),\quad
Q_5(t)=\frac{1}{4}Q_4'(t)-\Gamma_{31}(t),
\end{equation*}
are smooth functions depending only on $t$, which means that $M$ is
$5$-isoparametric. It completes the proof of the first part of
the theorem.

Now assume $N^{n+1}$ is locally a complex space form. Then $m=1$ in
the diagonalization (\ref{Jacobi rk one}) of $R(t)$, and $C_i$ in
the block decomposition (\ref{blocks of Si}) of $S(t)^i$ is a real
number (function) for each $i\geq1$. To prove the second part of the
theorem, \emph{i.e.}, a $3$-isoparametric hypersurface $M$ in
$N^{n+1}$ must be totally isoparametric, it suffices to prove the
following:


\begin{lem}\label{tr SiRSjR}
With notations and assumptions as above. If for some $i\geq1$,
$Q_1(t)$,$\cdots$,\\
$Q_{i+2}(t)$, $\Gamma_{11}(t)$,$\cdots$,$\Gamma_{i1}(t)$
are smooth functions depending only on $t$, then so are
$Q_{i+3}(t)$ and $\Gamma_{i+1, 1}(t)$.
\end{lem}

Since then by the inductive formula (\ref{inductive Qi general}),
the assumption that $M$ is $3$-isoparametric will imply
$Q_1(t),Q_2(t),Q_3(t)$ and $\Gamma_{11}(t)$ are smooth functions
depending only on $t$. So by this lemma we can show inductively that
$\Gamma_{k1}(t)$ and $Q_k(t)$ are smooth functions depending only on
$t$ for each $k\geq1$. Therefore, $M$ is totally isoparametric as
required.

\textbf{\emph{Proof of Lemma \ref{tr SiRSjR}}.} Under the
assumptions, it follows from (\ref{blocks}) that
$tr(A_1)$,$\cdots$,\\
$tr(A_i)$, $C_1,\cdots,C_i$ are smooth functions
depending only on $t$. Note that $S(t)^{j}S(t)^{i-1-j}=S(t)^{i-1}$
for each $0\leq j\leq i-1$. Substituting this into the block
decomposition (\ref{blocks of Si}), we get
$$A_jA_{i-1-j}+B_jB_{i-1-j}^t=A_{i-1},\quad C_jC_{i-1-j}+B_{j}^tB_{i-1-j}=C_{i-1}.$$
Hence,
\begin{equation*}
tr(B_jB_{i-1-j}^t)=tr(B_{j}^tB_{i-1-j})=C_{i-1}-C_jC_{i-1-j}
\end{equation*} is a
smooth function depending only on $t$, and so is
\begin{equation*}
tr(A_jA_{i-1-j})=tr(A_{i-1})-tr(B_jB_{i-1-j}^t).
\end{equation*} On the other
hand, a direct calculation shows that for each $0\leq j\leq i-1$,
\begin{eqnarray*}
&&\quad tr\Big(S(t)^jR(t)S(t)^{i-1-j}R(t)\Big)\\
&&=\kappa_1^2~tr(A_jA_{i-1-j})+2\kappa_1\kappa_2~tr(B_jB_{i-1-j}^t)+\kappa_2^2~C_jC_{i-1-j}
\end{eqnarray*}
is then a smooth function depending only on $t$. Therefore, by
(\ref{inductive Gamma general}),
$$\Gamma_{i+1, 1}(t)=\frac{1}{i}\Big\{\Gamma_{i1}'(t)-\sum_{j=0}^{i-1}tr\Big(S(t)^jR(t)S(t)^{i-1-j}R(t)\Big)\Big\}$$
is a smooth function depending only on $t$, and so is, by
(\ref{inductive Qi general}),
$$Q_{i+3}(t)=\frac{1}{i+2}Q_{i+2}'(t)-\Gamma_{i+1, 1}(t).$$ \hfill $\Box$

The proof of Theorem \ref{3isop rk one sp} is now complete. \hfill
$\Box$


\textbf{\emph{Proof of Theorem \ref{compatible rk one sp}.}} Use the
same notations as before. Let $M^n$ be a curvature-adapted
hypersurface in a locally rank one symmetric space $N^{n+1}$ of
non-constant sectional curvatures. Denote by $M_t$, $t\in(-\varepsilon,\varepsilon)$ nearby parallel hypersurfaces of
$M_0=M$. Evidently, each $M_t$ is curvature-adapted and the principal
orthonormal eigenvectors $\{e_i(t)|i=1,\cdots,n\}$ of $M_t$ can be
chosen to be parallel along normal geodesics such that under this
frame, the normal Jacobi operator $R(t)$ can be diagonalized as in
(\ref{Jacobi rk one}), and the shape operator $S(t)$ can be
diagonalized as
\begin{equation*}
S(t)=diag(\mu_1(t),\cdots,\mu_n(t)),
\end{equation*}
where $\mu_i(t)$'s are principal curvature functions of $M_t$ (cf.
\cite{Gr04}). Moreover, \begin{equation*}
S'(t)=diag(\mu_1'(t),\cdots,\mu_n'(t)),
\end{equation*}
and thus the Riccati equation (\ref{Riccati eq}) can be written as
\begin{equation}\label{Riccati eq for curv-adapted}
\mu_i'(t)=\mu_i(t)^2+\kappa_i, \quad for~~i=1,\cdots,n,
\end{equation}
where $\kappa_i=\kappa_1$ for $i\leq n-m$ and $\kappa_i=\kappa_2$
for $i>n-m$. Therefore, since $\kappa_i$'s are constant, the
principal curvatures $\mu_i(t)$'s of $M_t$,
$t\in(-\varepsilon,\varepsilon)$, are uniquely determined by initial
values $\mu_i(0)$'s, the principal curvatures of $M$. So when
$\mu_i(0)$'s are constant on $M$, $\mu_i(t)$'s are constant on $M_t$,
which completes the proof of the first part (i).

As for the second part (ii), the assumption that $M$ is
$1$-isoparametric implies that $Q_1(t)=\sum_{i=1}^n\mu_i(t)$ is a
smooth function depending only on $t$. Introduce two functions by
$$\Phi_i(t):=\sum_{p=1}^{n-m}\mu_{p}(t)^i,\quad \Psi_i(t):=\sum_{p=n-m+1}^{n}\mu_{p}(t)^i.$$
 Then we have
\begin{equation}\label{Qi Phi Psi}
Q_i(t)=\Phi_i(t)+\Psi_i(t),\quad
\Gamma_{i1}(t)=\kappa_1\cdot\Phi_i(t)+\kappa_2\cdot\Psi_i(t),
\end{equation}
and by (\ref{Riccati eq for curv-adapted}),
\begin{equation}\label{inductive Phi Psi}
\Phi_{i}'(t)=i(\Phi_{i+1}(t)+\kappa_1\Phi_{i-1}(t)),\quad
\Psi_{i}'(t)=i(\Psi_{i+1}(t)+\kappa_2\Psi_{i-1}(t)),
\end{equation}
and meanwhile, formula (\ref{inductive Qi general}) can be rewritten
as
\begin{equation}\label{inductive Qi rk one curv-adapted}
\frac{1}{i}Q_i'(t)=Q_{i+1}(t)+\kappa_1Q_{i-1}(t)+(\kappa_2-\kappa_1)\Psi_{i-1}(t).
\end{equation}
Taking the $k$-th derivative of $Q_1(t)$ with respective to $t$ by
(\ref{inductive Phi Psi}), (\ref{inductive Qi rk one curv-adapted})
inductively, we obtain
\begin{equation}\label{Q_1 k derivative}
\frac{1}{k!}Q_1^{(k)}(t)=Q_{k+1}(t)+\sum_{j=0}^{k-1}\Big(c_{kj}Q_j(t)+d_{kj}\Psi_j(t)\Big),
\end{equation}
where $c_{kj},d_{kj}$ are some constants depending only on the
indices and $\kappa_1$, $\kappa_2$. As $Q_1(t)$ is a smooth function
depending only on $t$, so is $Q_1^{(k)}(t)$, i.e., $Q_1^{(k)}(t)$ is
constant on $M_t$ for each $k\geq0$. Fixing
$t\in(-\varepsilon,\varepsilon)$ in (\ref{Q_1 k derivative}), then
it follows that the principal curvatures $\mu_1(t),\cdots,\mu_n(t)$
of $M_t$ are solutions of the algebraic equations
\begin{equation}\label{Pk}
P_{k+1}(x_1,\cdots,x_n):=\rho_{k+1}(x_1,\cdots,x_n)+\widehat{P}_k(x_1,\cdots,x_n)=0,\quad
for~~k=0,\cdots,n-1,
\end{equation}
where $\rho_j(x_1,\cdots,x_n):=\sum_{i=1}^nx_i^j$ is the $j$-th
power sum over the variables $(x_1,\cdots,x_n)$ as in (\ref{rhok}), while
\begin{equation*}
\widehat{P}_k(x_1,\cdots,x_n):=\sum_{j=0}^{k-1}\Big(c_{kj}~\rho_j(x_1,\cdots,x_n)+d_{kj}~\rho_j(x_{n-m+1},\cdots,x_n)\Big)-\frac{1}{k!}Q_1^{(k)}(t)
\end{equation*}
is a polynomial of degree less than $k$ with constant coefficients
for $k\geq1$ and $\widehat{P}_0:=-Q_1(t)$ is a constant.

Finally, the case $n\leq2$ is not possible, since $N^{n+1}$
is a locally rank one symmetric space of non-constant sectional
curvature. For $n\geq3$, we
can not derive directly from (\ref{inductive Qi rk one
curv-adapted}) that $Q_i(t)$'s or $\mu_i(t)$'s are constant on $M_t$.
However, making use of the following lemma and (\ref{Pk}), we know that
$(\mu_1(t),\cdots,\mu_n(t))$ belongs to a finite subset of
$\mathbb{C}^n$ and thus $\mu_i(t)$'s are constant on $M_t$ since
$M_t$ is connected. It means that $M$ is totally isoparametric.

The proof is now complete. \hfill $\Box$


Now, we have to state explicitly the lemma on algebraic geometry
used above.
\begin{lem}\label{zero locus}For each $n\geq1$,
define polynomials $P_k\in\mathbb{C}[x_1,\cdots,x_n]$ by
$$P_k:=\rho_{k}(x_1,\cdots,x_n)+\widetilde{P}_{k-1}(x_1,\cdots,x_n),\quad
for ~~k=1,\cdots,n,$$ where $\rho_{k}$ is the $k$-th power sum
polynomial as before, $\widetilde{P}_{k-1}$ is an arbitrary
polynomial of degree less than $k$. Then $P_1,\cdots,P_n$ form a
regular sequence in $\mathbb{C}[x_1,\cdots,x_n]$. Consequently, the
dimension of each variety $V_k$ in $\mathbb{C}^n$ defined by
$P_1=\cdots=P_k=0$ is less than or equal to $n-k$ for
$k=1,\cdots,n$. In particular, $V_{n}$ is a finite subset of
$\mathbb{C}^n$.
\end{lem}
\begin{proof}
First recall (cf. \cite{E}, \cite{Mat}) that a sequence $r_1,\cdots,r_k$
in a commutative ring $\mathcal {R}$ with identity is called a
\emph{regular sequence} if $(1)$ the ideal $(r_1,\cdots,r_k)\neq \mathcal
{R}$; $(2) $ $r_1$ is not a zero divisor in $\mathcal {R}$; and $(3)$ $r_{i+1}$
is not a zero divisor in the quotient ring $\mathcal
{R}/(r_1,\cdots,r_i)$ for $i=1,\cdots,k-1$.

Now we will work on the polynomial ring $\mathcal {R}=\mathbb{C}[x_1,\cdots,x_n]$. Obviously, it is
a \emph{Cohen-Macaulay} ring, possessing the property
that $dim(\mathcal {R}/(P_1,\cdots,P_k))=n-k$ for a regular sequence
$P_1,\cdots,P_k$ in  $\mathcal {R}$. Meanwhile, we know that
$dim(V_k)=dim(\mathcal {R}/I(V_k))$, where $I(V_k)\supset
(P_1,\cdots,P_k)$ is the ideal of the variety $V_k$. Therefore, when
$P_1,\cdots,P_n$ form a regular sequence, $dim(V_k)\leq n-k$ for
$k=1,\cdots,n$. In particular, $dim(V_n)=0$. The last assertion is due to
the facts that every variety in $\mathbb{C}^n$ can be expressed as a
union of finite irreducible varieties and that a zero-dimensional
irreducible variety in $\mathbb{C}^n$ is just a point. To complete the proof
of the lemma, it suffices to show that the polynomials $P_1,\cdots,P_n$ form a regular sequence in
$\mathcal {R}$.

Obviously, $P_1$ forms a regular sequence in $\mathcal {R}$. Suppose that
$P_1,\cdots,P_n$ do not form a regular sequence, there exists some $k$ with
$1\leq k<n$ such that $P_{k+1}$ is a zero divisor modulo
$(P_1,\cdots,P_k)$ in $\mathcal {R}$. Then we may choose a relation
of minimal degree of the form
\begin{equation}\label{minimal relation}
f_1P_1+\cdots+f_{k+1}P_{k+1}=0,
\end{equation}
 where
$f_1,\cdots,f_{k+1}$ are polynomials of minimal degrees modulo
$(P_1,\cdots,P_k)$. Denote by $D(>0)$ the maximal degree of
$f_iP_i$'s. Let $f_{i_1}P_{i_1},\cdots,f_{i_r}P_{i_r}$ be those of
maximal degree $D$ for some $1\leq i_1<\cdots<i_r\leq k+1$. Then one
can pick out the homogeneous components
$\tilde{f}_{i_1}\rho_{i_1},\cdots,\tilde{f}_{i_r}\rho_{i_r}$ of
maximal degree from them in equation (\ref{minimal relation}) such
that
\begin{equation}\label{maximal degree}
\tilde{f}_{i_1}\rho_{i_1}+\cdots+\tilde{f}_{i_r}\rho_{i_r}=0,
\end{equation}
where $\tilde{f}_{i_1},\cdots,\tilde{f}_{i_r}$ are the homogeneous
components of maximal degrees of $f_{i_1},\cdots,f_{i_r}$, respectively. Recall a
well known fact that the power sum polynomials
$\rho_1,\cdots,\rho_n$ form a regular sequence in $\mathcal {R}$
(\cite{Sm}). Then by (\ref{maximal degree}), $r>1$ and
$\tilde{f}_{i_r}\in(\rho_1,\cdots,\rho_{i_r-1})$, which imply that
there exist homogeneous polynomials $a_1,\cdots,a_{i_r-1}$ such that
$$\tilde{f}_{i_r}=a_1\rho_1+\cdots+a_{i_r-1}\rho_{i_r-1},$$
and therefore,
$$f_{i_r}=a_1P_1+\cdots+a_{i_r-1}P_{i_r-1}+\hat{f}_{i_r}\equiv \hat{f}_{i_r}, \quad mod~~ (P_1,\cdots,P_k),$$
where $\hat{f}_{i_r}$ is a polynomial of degree less than
$D-i_r=deg(f_{i_r})$, which contradicts the original choice of
minimal relation (\ref{minimal relation}).

The proof is now complete.
\end{proof}


We conclude this section with a brief proof of Remark \ref{Osserman}.

\textbf{\emph{Proof of Remark \ref{Osserman}}.} Suppose that the ambient
manifold $N^{n+1}$ is an Osserman manifold. Its Jacobi operator
$K_{\xi}$, by definition, has constant eigenvalues independent of
$\xi$ and points all over $N$. This property guarantees that the
normal Jacobi operator $R(t)=K_{\nu_t}$ of the parallel hypersurface
$M_t$ in $N^{n+1}$ has constant eigenvalues
$\kappa_1,\cdots,\kappa_n$ for any
$t\in(-\varepsilon,~\varepsilon)$, though the covariant derivative
$R'(t):=\nabla_{\nu_t}R(t)$ along normal direction $\nu_t$ might not
vanish, different from the case in a locally rank one symmetric space. Further,
we suppose that each $M_t$ in $N$ is curvature-adapted, that is, the
shape operator $S(t)$ and the normal Jacobi operator $R(t)$ are
simultaneously diagonalizable, which is automatically satisfied for
a curvature-adapted hypersurface in a locally symmetric space.
Therefore, the assertion in Remark \ref{Osserman} actually does
nothing but abandon the assumption $R'(t)=0$ in Theorem
\ref{compatible rk one sp}.

Let $\epsilon_1(t),\cdots,\epsilon_n(t)$ be a local orthonormal
frame of $M_t$ smoothly depending on $t$ such that they are
eigenvectors of $R(t)$ and $S(t)$ at the same time, corresponding to eigenvalues
$\kappa_1,\cdots,\kappa_n$ and $\mu_1(t),\cdots,\mu_n(t)$, respectively. Then under this frame,
$$\langle S'(t)\epsilon_i(t),\epsilon_i(t)\rangle=\mu_i(t)^2+\kappa_i.$$
The left side of the preceding equation can be deduced to
\begin{eqnarray}
\langle
S'(t)\epsilon_i(t),\epsilon_i(t)\rangle&=&\Big\langle\nabla_{\nu_t}(S(t)\epsilon_i(t))-S(t)\nabla_{\nu_t}\epsilon_i(t),~\epsilon_i(t)\Big\rangle\nonumber\\
&=&\mu_i'(t)+\Big\langle\Big(\mu_i(t)I-S(t)\Big)~\nabla_{\nu_t}\epsilon_i(t),~\epsilon_i(t)\Big\rangle\nonumber\\&=&\mu_i'(t).\nonumber
\end{eqnarray}
Hence, we still have the Riccati equation (\ref{Riccati eq for
curv-adapted}) in this case:
$$\mu_i'(t)=\mu_i(t)^2+\kappa_i.$$
With this equality, we are able to complete the proof of Remark
(\ref{Osserman}). For this purpose, we distinguish two cases. First,
if $M$ has constant principal curvatures, it follows directly from
the identity above that $M$ is totally isoparametric; Next, if $M$
is $1$-isoparametric, we can also take the $k$-th derivative of
$Q_1(t)$ to obtain a sequence of algebraic equations $P_k=0$,
$k=0,1,\cdots,n-1$, for $(\mu_1(t),\cdots,\mu_n(t))$ similar to
(\ref{Pk}). By means of Lemma \ref{zero locus}, we know that
$\mu_1(t),\cdots,\mu_n(t)$ are constant on $M_t$ and thus $M$ is
totally isoparametric, as desired.\hfill $\Box$

\begin{ack}
The authors would like to thank Professors Q.S. Chi, C.K. Peng and
G. Thorbergsson for their valuable suggestions and helpful comments
during the preparation of this paper. The authors also thank the referees for their useful suggestions.
\end{ack}

\end{document}